\documentclass[11pt]{article}
\usepackage{epsf}
\usepackage{amsmath}
\usepackage{epsfig}
\usepackage{times}
\usepackage{amssymb}
\usepackage{amsthm}
\usepackage{setspace}
\usepackage{cite}

\usepackage{algorithmic}  
\usepackage{algorithm}

\usepackage{shadow}
\usepackage{fancybox}
\usepackage{fancyhdr}

\def\y{{\bf y}}

\def\x{{\bf x}}

\def\x{{\mathbf x}}

\def\x{{\bf x}}
\def\y{{\bf y}}

\def\h{{\bf h}}

\def\be{\begin{equation}}
\def\ee{\end{equation}}
\def\ba{\left[\begin{array}}
\def\ea{\end{array}\right]}

\def\x{{\bf x}}
\def\y{{\bf y}}

\def\1{{\bf 1}}

\def\g{{\bf g}}
\def\0{{\bf 0}}

\newtheorem{theorem}{Theorem}

\newtheorem{lemma}{Lemma}

\begin{document}

\begin{singlespace}

\title {Discrete perceptrons 
}
\author{
\textsc{Mihailo Stojnic}
\\
\\
{School of Industrial Engineering}\\
{Purdue University, West Lafayette, IN 47907} \\
{e-mail: {\tt mstojnic@purdue.edu}} }
\date{}
\maketitle

\centerline{{\bf Abstract}} \vspace*{0.1in}

Perceptrons have been known for a long time as a promising tool within the neural networks theory. The analytical treatment for a special class of perceptrons started in seminal work of Gardner \cite{Gar88}. Techniques initially employed to characterize perceptrons relied on a statistical mechanics approach. Many of such predictions obtained in \cite{Gar88} (and in a follow-up \cite{GarDer88}) were later on established rigorously as mathematical facts (see, e.g. \cite{SchTir02,SchTir03,TalBook,StojnicGardGen13,StojnicGardSphNeg13,StojnicGardSphErr13}). These typically related to spherical perceptrons. A lot of work has been done related to various other types of perceptrons. Among the most challenging ones are what we will refer to as the discrete perceptrons. An introductory statistical mechanics treatment of such perceptrons was given in \cite{GutSte90}. Relying on results of \cite{Gar88}, \cite{GutSte90} characterized many of the features of several types of discrete perceptrons. We in this paper, consider a similar subclass of discrete perceptrons and provide a mathematically rigorous set of results related to their performance. As it will turn out, many of the statistical mechanics predictions obtained for discrete predictions will in fact appear as mathematically provable bounds. This will in a way emulate a similar type of behavior we observed in \cite{StojnicGardGen13,StojnicGardSphNeg13,StojnicGardSphErr13} when studying spherical perceptrons.

\vspace*{0.25in} \noindent {\bf Index Terms: Discrete perceptrons; storage capacity}.

\end{singlespace}

\section{Introduction}
\label{sec:back}

In last several decades there has been a lot of great work related to an analytical characterization of neural networks performance. While the neural networks have been known for quite some time it is probably with the appearance of powerful statistical mechanics techniques that incredibly results related to characterization of their performance started appearing. Of course, since the classical perceptrons are among the simplest and most fundamental tools within the frame of neural networks theory, it is a no surprise that among the very first analytical characterizations were the ones related to them. Probably the most successful one and we would say the most widely known one is the seminal approach of Gardner, developed in \cite{Gar88} and complemented in a follow-up \cite{GarDer88}. There, Gardner adapted by that time already well-known replica approach so that it can treat almost any feature of various perceptron models. She started the story of course with probably the simplest possible case, namely the spherical perceptron. Then in \cite{Gar88} she and in \cite{GarDer88} she and Derrida proceeded with fairly accurate predictions/approximations for its storage capacities in several different scenarios: positive thresholds (we will often referred to such perceptrons as the positive spherical perceptrons), negative thresholds, correlated/uncorrelated patterns, patterns stored incorrectly and many others.

While these predictions were believed to be either exact in some cases or fairly good approximations in others, they remained quite a mathematical challenge for a long time. Somewhat paradoxically, one may though say that the first successful confirmation of some of the results from \cite{Gar88,GarDer88} had actually arrived a long time before they appeared. Namely, for a special case of spherical perceptrons with zero-thresholds, the storage capacity was already known either explicitly within the neural networks community or within pure mathematics (see, e.g. \cite{Schlafli,Cover65,Winder,Winder61,Wendel62,Cameron60,Joseph60,BalVen87,Ven86}). However, the first real confirmation of the complete treatment presented in \cite{Gar88} appeared in \cite{SchTir02,SchTir03}. There the authors were able to confirm the predictions made in \cite{Gar88} related to the storage capacity of the positive spherical perceptrons. Moreover, they confirmed that the prediction related to the volume of the bond strengths that satisfies the perceptron dynamics presented in \cite{Gar88} is also correct. Later on, in \cite{TalBook} Talagrand reconfirmed these predictions through a somewhat different approach. In our own work \cite{StojnicGardGen13} we also presented a simple framework that can be used to confirm many of the storage capacity predictions made in \cite{Gar88}. Moreover, in \cite{StojnicGardSphNeg13} we confirmed that the results presented in \cite{Gar88} related to the negative spherical perceptrons are rigorous upper bounds that in certain range of problem parameters may even be lowered. Along the same lines we then in \cite{StojnicGardSphErr13} attacked a bit harder spherical perceptron type of problem that relates to their functioning as erroneous storage memories. This problem was initially treated in \cite{GarDer88} through an extension of the replica approach utilized in \cite{Gar88}. The predictions obtained based on such an approach were again proved as rigorous upper bounds in \cite{StojnicGardSphErr13}. Moreover, \cite{StojnicGardSphErr13} hinted that while the predictions made in \cite{GarDer88} are rigorous upper bounds one may even be able to lower them in certain range of parameters of interest.

Of course, as one may note all the above mentioned initial treatments relate to the so-called spherical perceptrons. These are long believed to be substantially easier for an analytical treatment than some other classes of perceptrons. On the other hand, we believe that among the most difficult for an analytical treatment are the ones that we will call discrete perceptrons. While we will below give a detailed description of what we will mean by discrete perceptrons, we would like to just mention here that an introductory treatment of such perceptrons was already started in \cite{Gar88,GarDer88}. There it was demonstrated that framework designed to cover the spherical perceptron can in fact be used to obtain predictions for many other perceptrons as well and among them certainly for what we will call $\pm 1$ discrete perceptrons. However, as already observed in \cite{Gar88,GarDer88} it may happen that the treatment of such perceptrons may be substantially more difficult than the spherical ones. To be a bit more specific, an initial set of results obtained for the storage capacity in the simple zero-thresholds case indicated that the variant of the framework given in \cite{Gar88} may not be able to match even the simple combinatorial results one can obtain in such a case. As a result it was hinted that a more advanced version of the framework from \cite{Gar88} may be needed. In \cite{GutSte90} the authors went a bit further and considered various other types of discrete perceptrons. For many of them they were able to provide a similar set of predictions given in \cite{Gar88} for the simple spherical and $\pm 1$ ones. Moreover, they hinted at a potential way that can be used to bridge some of deficiencies that the predictions given in \cite{Gar88} may have. In this paper we will also study several discrete perceptrons. On top of that we will cover a ``not so discrete case" which in a sense is a limiting case of some of the discrete cases studied in \cite{GutSte90} and itself was also studied in \cite{GutSte90}. The framework that we will present will rigorously confirm that the results related to these classes of perceptrons obtained in \cite{GutSte90} relying on the replica symmetry approach of \cite{Gar88} are in fact rigorous upper bounds on their true values. For the above mentioned ``not so discrete case" it will turn out that the predictions made in \cite{GutSte90} can in fact be proven as exact.

Before going into the details of our approach we will recall on the basic definitions related to the perceptrons and needed for its analysis. Also, to make the presentation easier to follow we find it useful to briefly sketch how the rest of the paper is organized. In Section \ref{sec:mathsetupper} we will, as mentioned above, introduce a more formal mathematical description of how a perceptron operates. Along the same lines we will formally present the sevearl classes/types of perceptrons that we will study in later sections. In Section \ref{sec:knownres} we will present several results that are known for the classical spherical perceptron as some of them we will actually need to establish the main results of this paper as well. In Sections \ref{sec:discper}, \ref{sec:01per}, and \ref{sec:boxper} we will discuss the three types of perceptrons that we plan to study in great detail in this paper. Finally, in Section \ref{sec:conc} we will discuss obtained results and present several concluding remarks.

\section{Perceptrons as mathematical problems}
\label{sec:mathsetupper}

To make this part of the presentation easier to follow we will try to introduce all important features of the perceptron that we will need here by closely following what was done in \cite{Gar88} (and for that matter in our recent work \cite{StojnicGardGen13,StojnicGardSphNeg13,StojnicGardSphErr13}). So, as in \cite{Gar88}, we start with the following dynamics:
\begin{equation}
H_{ik}^{(t+1)}=\mbox{sign} (\sum_{j=1,j\neq k}^{n}H_{ij}^{(t)}X_{jk}-T_{ik}).\label{eq:defdyn}
\end{equation}
Following \cite{Gar88} for any fixed $1\leq i\leq m$ we will call each $H_{ij},1 \leq j\leq n $, the icing spin, i.e. $H_{ij}\in\{-1,1\},\forall i,j$. Continuing further with following \cite{Gar88}, we will call $X_{jk},1\leq j\leq n$, the interaction strength for the bond from site $j$ to site $i$. To be in a complete agreement with \cite{Gar88}, we in (\ref{eq:defdyn}) also introduced quantities $T_{ik},1\leq i\leq m,1\leq k\leq n$. $T_{ik}$is typically called the threshold for site $k$ in pattern $i$. However, to make the presentation easier to follow, we will typically assume that $T_{ik}=0$. Without going into further details we will mention though that all the results that we will present below can be easily modified so that they include scenarios where $T_{ik}\neq 0$.

Now, the dynamics presented in (\ref{eq:defdyn}) works by moving from a $t$ to $t+1$ and so on (of course one assumes an initial configuration for say $t=0$). Moreover, the above dynamics will have a fixed point if say there are strengths $X_{jk},1\leq j\leq n,1\leq k\leq m$, such that for any $1\leq i\leq m$
\begin{eqnarray}
& & H_{ik}\mbox{sign} (\sum_{j=1,j\neq k}^{n}H_{ij}X_{jk}-T_{ik})=1\nonumber \\
& \Leftrightarrow & H_{ik}(\sum_{j=1,j\neq k}^{n}H_{ij}X_{jk}-T_{ik})>0,1\leq j\leq n,1\leq k\leq n.\label{eq:defdynfp}
\end{eqnarray}
Of course, the above is a well known property of a very general class of dynamics. In other words, unless one specifies the interaction strengths the generality of the problem essentially makes it easy. After considering the general scenario introduced above, \cite{Gar88} then proceeded and specialized it to a particular case which amounts to including spherical restrictions on $X$. A more mathematical description of such restrictions considered in \cite{Gar88} essentially boils down to the following constraints
\begin{equation}
\sum_{j=1}^{n}X_{ji}^2=1,1\leq i\leq n.\label{eq:cosntXsph}
\end{equation}
These were of course the same restrictions/constraints considered in a series of our own work \cite{StojnicGardGen13,StojnicGardSphErr13,StojnicGardSphNeg13}. In this paper however, we will focus on a set of what we will call discrete restirctions/conctraints. While the methods that we will present below will be fairly powerful to handle many different discrete restrictions we will to avoid an overloading and for clarity purposes here present the following two types of discrete constraints.
\begin{eqnarray}
X_{ji}& \in & \left \{-\frac{1}{\sqrt{n}},\frac{1}{\sqrt{n}}\right \},1\leq i\leq n,1\leq j\leq m\nonumber \\
X_{ji}& \in & \left \{0,\frac{1}{\sqrt{n}}\right \},1\leq i\leq n,1\leq j\leq m.\label{eq:cosntXdisc}
\end{eqnarray}
We will call the perceptron operating with the first set of constraints given in (\ref{eq:cosntXdisc}) the $\pm 1$ perceptron (in fact we may often refer to the bond strengths $X$ in such a perceptron as the ones from $\{-1,1\}$ set although for the scaling purposes we assumed the above more convenient $\left \{-\frac{1}{\sqrt{n}},\frac{1}{\sqrt{n}}\right \}$ set). Analogously, we will call the perceptron operating with the second set of constraints given in (\ref{eq:cosntXdisc}) the $0/1$ perceptron. Moreover, we will also consider a third type of the perceptron that operates with the following constraints on the bond strengths
\begin{equation}
X_{ji}\in [-\frac{1}{\sqrt{n}},\frac{1}{\sqrt{n}}],1\leq i\leq n,1\leq j\leq m.\label{eq:cosntXbox}
\end{equation}
We will refer to the perceptron operating with the set of constraints given in (\ref{eq:cosntXbox}) the box-constrained perceptron.

The fundamental question that one typically considers then is the so-called storage capacity of the above dynamics or alternatively a neural network that it would represent (of course this is exactly one of the questions considered in \cite{Gar88}). Namely, one then asks how many patterns $m$ ($i$-th pattern being $H_{ij},1\leq j\leq n$) one can store so that there is an assurance that they are stored in a stable way. Moreover, since having patterns being fixed points of the above introduced dynamics is not enough to insure having a finite basin of attraction one often may impose a bit stronger threshold condition
\begin{eqnarray}
& & H_{ik}\mbox{sign} (\sum_{j=1,j\neq k}^{n}H_{ij}X_{jk}-T_{ik})=1\nonumber \\
& \Leftrightarrow & H_{ik}(\sum_{j=1,j\neq k}^{n}H_{ij}X_{jk}-T_{ik})>\kappa,1\leq j\leq n,1\leq k\leq n,\label{eq:defdynfpstr}
\end{eqnarray}
where typically $\kappa$ is a positive number. We will refer to a perceptron governed by the above dynamics and coupled with the spherical restrictions and a positive threshold $\kappa$ as the positive spherical perceptron (alternatively, when $\kappa$ is negative we would refer to it as the negative spherical perceptron; for such a perceptron and resulting mathematical problems/results see e.g. \cite{StojnicGardSphNeg13}).

Also, we should mentioned that beyond the above mentioned cases many other variants of the neural network models are possible from a purely mathematical perspective. Moreover, many of them have found applications in various other fields as well. For example, a nice set of references that contains a collection of results related to various aspects of different neural networks models and their bio- and many other applications is  \cite{AgiAnnBarCooTan13a,AgiAnnBarCooTan13b,AgiBarBarGalGueMoa12,AgiBarGalGueMoa12,AgiAstBarBurUgu12,StojnicAsymmLittBnds11,BruParRit92}. We should also mention that while we chose here a particular set of neural network models, the results that we will present below can be adapted to be of use in pretty much any other known model. Our goal here is to try to keep the presentation somewhat self-contained, clear, and without too much of overloading. Because of that we selected only a small number of cases for which we will present the concrete results. A treatment of many others we will present elsewhere.

\section{Known results}
\label{sec:knownres}

As mentioned above, our main interest in this paper will be studying what we call discrete perceptrons. However, many of the results that we will present will lean either conceptually or even purely analytically on many results that we created for the so-called spherical perceptrons. In fact, quite a few technical details that we will need here we already needed when treating various aspects of the spherical perceptrons, see e.g. \cite{StojnicGardSphNeg13,StojnicGardGen13,StojnicGardSphErr13}. In that sense we will find it useful to have quite handy some of the well-known spherical perceptron results readily available. So, before proceeding with the problems that we will study here in great detail we will first recall on several results known for the standard spherical perceptron.

In the first of the subsections below we will hence look at the spherical perceptrons, and in the following one we will then present a few results known for the discrete perceptrons. That way it will also be easier to later on properly position the results we intend to present here within the scope of what is already known.

\subsection{Spherical perceptron}
\label{sec:sphper}

We should preface this brief presentation of the known results by mentioning that a way more is known that what we will present below. However, we will restrict ourselves to the facts that we deem are of most use for the presentation that will follow in later sections.

\subsubsection{Statistical mechanics}
\label{sec:statmech}

We of course start with recalling on what was presented in \cite{Gar88}. In \cite{Gar88} a replica type of approach was designed and based on it a characterization of the storage capacity was presented. Before showing what exactly such a characterization looks like we will first formally define it. Namely, throughout the paper we will assume the so-called linear regime, i.e. we will consider the so-called \emph{linear} scenario where the length and the number of different patterns, $n$ and $m$, respectively are large but proportional to each other. Moreover, we will denote the proportionality ratio by $\alpha$ (where $\alpha$ obviously is a constant independent of $n$) and will set
\begin{equation}
m=\alpha n.\label{eq:defmnalpha}
\end{equation}
Now, assuming that $H_{ij},1\leq i\leq m,1\leq j\leq n$, are i.i.d. symmetric Bernoulli random variables, \cite{Gar88}, using the replica approach, gave the following estimate for $\alpha$ so that (\ref{eq:defdynfpstr}) holds with overwhelming probability (under overwhelming probability we will in this paper assume a probability that is no more than a number exponentially decaying in $n$ away from $1$)
\begin{equation}
\alpha_c(\kappa)=(\frac{1}{\sqrt{2\pi}}\int_{-\kappa}^{\infty}(z+\kappa)^2e^{-\frac{z^2}{2}}dz)^{-1}.\label{eq:garstorcap}
\end{equation}
Based on the above characterization one then has that $\alpha_c$ achieves its maximum over positive $\kappa$'s as $\kappa\rightarrow 0$. One in fact easily then has
\begin{equation}
\lim_{\kappa\rightarrow 0}\alpha_c(\kappa)=2.\label{eq:garstorcapk0}
\end{equation}
Also, to be completely exact, in \cite{Gar88}, it was predicted that the storage capacity relation from (\ref{eq:garstorcap}) holds for the range $\kappa\geq 0$.

\subsubsection{Rigorous results -- positive spherical perceptron ($\kappa\geq 0$)}
\label{sec:posphper}

The result given in (\ref{eq:garstorcapk0}) is of course well known and has been rigorously established either as a pure mathematical fact or even in the context of neural networks and pattern recognition \cite{Schlafli,Cover65,Winder,Winder61,Wendel62,Cameron60,Joseph60,BalVen87,Ven86}. In a more recent work \cite{SchTir02,SchTir03,TalBook} the authors also considered the storage capacity of the spherical perceptron and established that when $\kappa\geq 0$ (\ref{eq:garstorcap}) also holds. In our own work \cite{StojnicGardGen13} we revisited the storage capacity problems and presented an alternative mathematical approach that was also powerful enough to reestablish the storage capacity prediction given in (\ref{eq:garstorcap}). We below formalize the results obtained in \cite{SchTir02,SchTir03,TalBook,StojnicGardGen13}.

\begin{theorem} \cite{SchTir02,SchTir03,TalBook,StojnicGardGen13}
Let $H$ be an $m\times n$ matrix with $\{-1,1\}$ i.i.d.Bernoulli components. Let $n$ be large and let $m=\alpha n$, where $\alpha>0$ is a constant independent of $n$. Let $\alpha_c$ be as in (\ref{eq:garstorcap}) and let $\kappa\geq 0$ be a scalar constant independent of $n$. If $\alpha>\alpha_c$ then with overwhelming probability there will be no $\x$ such that $\|\x\|_2=1$ and (\ref{eq:defdynfpstr}) is feasible. On the other hand, if $\alpha<\alpha_c$ then with overwhelming probability there will be an $\x$ such that $\|\x\|_2=1$ and (\ref{eq:defdynfpstr}) is feasible.
\label{thm:SchTirTalSto}
\end{theorem}
\begin{proof}
Presented in various forms in \cite{SchTir02,SchTir03,TalBook,StojnicGardGen13}.
\end{proof}

As mentioned earlier, the results given in the above theorem essentially settle the storage capacity of the positive spherical perceptron or the Gardner problem in a statistical sense (it is rather clear but we do mention that the overwhelming probability statement in the above theorem is taken with respect to the randomness of $H$). However, they strictly speaking relate only to the \emph{positive} spherical perceptron. It is not clear if they would automatically translate to the case of the negative spherical perceptron. As we hinted earlier, the case of the negative spherical perceptron ($\kappa<0$) may be more of interest from a purely mathematical point of view than it is from say the neural networks point of view. Nevertheless, such a mathematical problem may turn out to be a bit harder than the one corresponding to the standard positive case. In fact, in \cite{TalBook}, Talagrand conjectured (conjecture 8.4.4) that the above mentioned $\alpha_c$ remains an upper bound on the storage capacity even when $\kappa<0$, i.e. even in the case of the negative spherical perceptron. Such a conjecture was confirmed in our own work \cite{StojnicGardGen13}. In the following subsection we will briefly summarize what in fact was shown in \cite{StojnicGardGen13}.

\begin{figure}[htb]
\centering
\centerline{\epsfig{figure=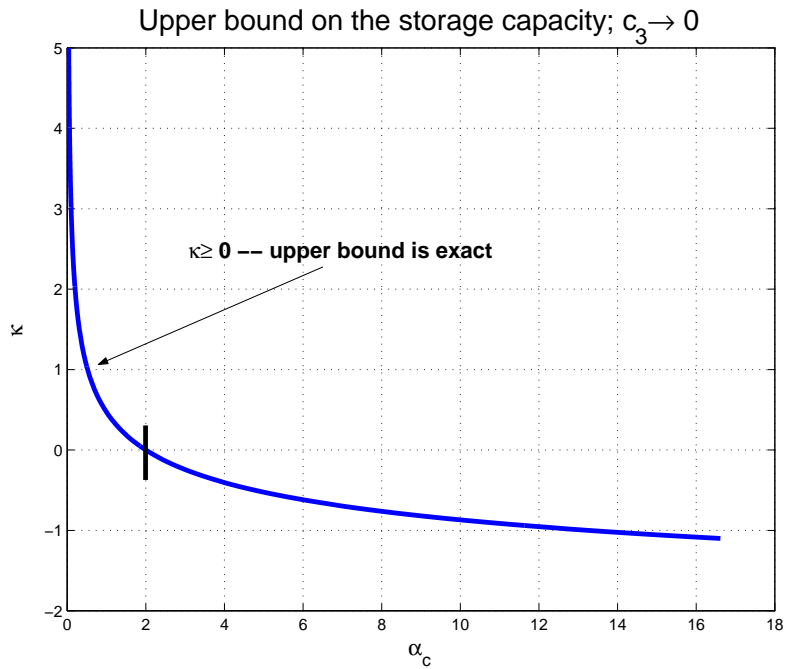,width=10.5cm,height=9cm}}
\caption{$\kappa$ as a function of $\alpha$}
\label{fig:alfackappa}
\end{figure}

\subsubsection{Rigorous results -- negative spherical perceptron ($\kappa< 0$)}
\label{sec:rigresnegsphper}

In our recent work \cite{StojnicGardSphNeg13} we went a step further and considered the negative version of the standard spherical perceptron. While the results that we will present later on in Sections \ref{sec:discper}, \ref{sec:01per}, and \ref{sec:boxper} will relate to any $\kappa$ our main concern will be from a neural network point of view and consequently the emphasis will be on the positive case, i.e. to $\kappa\geq 0$ scenario. Still, in our own view the results related to the negative spherical perceptron are important as they hint that already in the spherical case things may not be as easy as they may seem to be based on the results of \cite{Gar88,SchTir02,SchTir03,TalBook,StojnicGardGen13} for the positive spherical perceptron.

Moreover, a few technical details needed for presenting results in later section were already observed in \cite{StojnicGardGen13,StojnicGardSphNeg13} and we find it convenient to recall on them while at the same time revisiting the negative spherical perceptron. This will in our view substantially facilitate the exposition that will follow.

We first recall that in \cite{StojnicGardSphNeg13} we studied the so-called uncorrelated case of the spherical perceptron (more on an equally important correlated case can be found in e.g. \cite{StojnicGardGen13,Gar88}). This is the same scenario that we will study here (so the simplifications that we made in \cite{StojnicGardSphNeg13} and that we are about to present below will be in place later on as well). In the uncorrelated case, one views all patterns $H_{i,1:n},1\leq i\leq m$, as uncorrelated (as expected, $H_{i,1:n}$ stands for vector $[H_{i1},H_{i2},\dots,H_{in}]$). Now, the following becomes the corresponding version of the question of interest mentioned above: assuming that $H$ is an $m\times n$ matrix with i.i.d. $\{-1,1\}$ Bernoulli entries and that $\|\x\|_2=1$, how large $\alpha=\frac{m}{n}$ can be so that the following system of linear inequalities is satisfied with overwhelming probability
\begin{equation}
H\x\geq \kappa.\label{eq:defprobucor}
\end{equation}
This of course is the same as if one asks how large $\alpha$ can be so that the following optimization problem is feasible with overwhelming probability
\begin{eqnarray}
& & H\x\geq \kappa\nonumber \\
& & \|\x\|_2=1.\label{eq:defprobucor1}
\end{eqnarray}
To see that (\ref{eq:defprobucor}) and (\ref{eq:defprobucor1}) indeed match the above described fixed point condition it is enough to observe that due to statistical symmetry one can assume $H_{i1}=1,1\leq i\leq m$. Also the constraints essentially decouple over the columns of $X$ (so one can then think of $\x$ in (\ref{eq:defprobucor}) and (\ref{eq:defprobucor1}) as one of the columns of $X$). Moreover, the dimension of $H$ in (\ref{eq:defprobucor}) and (\ref{eq:defprobucor1}) should be changed to $m\times (n-1)$; however, since we will consider a large $n$ scenario to make writing easier we keep the dimension as $m\times n$. Also, as mentioned to a great extent in \cite{StojnicGardGen13,StojnicGardSphErr13,StojnicGardSphErr13}, we will, without a loss of generality, treat $H$ in (\ref{eq:defprobucor1}) as if it has i.i.d. standard normal components. Moreover, in \cite{StojnicGardGen13} we also recognized that (\ref{eq:defprobucor1}) can be rewritten as the following optimization problem
\begin{eqnarray}
\xi_n=\min_{\x} \max_{\lambda\geq 0} & &  \kappa\lambda^T\1- \lambda^T H\x \nonumber \\
\mbox{subject to} & & \|\lambda\|_2= 1\nonumber \\
& & \|\x\|_2=1,\label{eq:uncorminmax}
\end{eqnarray}
where $\1$ is an $m$-dimensional column vector of all $1$'s. Clearly, if $\xi_n\leq 0$ then (\ref{eq:defprobucor1}) is feasible. On the other hand, if $\xi_n>0$ then (\ref{eq:defprobucor1}) is not feasible. That basically means that if we can probabilistically characterize the sign of $\xi_n$ then we could have a way of determining $\alpha$ such that $\xi_n\leq 0$. That is exactly what we have done in \cite{StojnicGardGen13} on an ultimate level for $\kappa\geq 0$ and on a say upper-bounding level for $\kappa<0$. Relying on the strategy developed in \cite{StojnicRegRndDlt10,StojnicGorEx10} and on a set of results from \cite{Gordon85,Gordon88} we in \cite{StojnicGardGen13} proved the following theorem that essentially extends Theorem \ref{thm:SchTirTalSto} to the $\kappa<0$ case and thereby resolves Conjecture 8.4.4 from \cite{TalBook} in positive:

\begin{theorem} \cite{StojnicGardGen13}
Let $H$ be an $m\times n$ matrix with i.i.d. standard normal components. Let $n$ be large and let $m=\alpha n$, where $\alpha>0$ is a constant independent of $n$. Let $\xi_n$ be as in (\ref{eq:uncorminmax}) and let $\kappa$ be a scalar constant independent of $n$. Let all $\epsilon$'s be arbitrarily small constants independent of $n$. Further, let $\g_i$ be a standard normal random variable and set
\begin{equation}
f_{gar}(\kappa)=\frac{1}{\sqrt{2\pi}}\int_{-\kappa}^{\infty}(\g_i+\kappa)^2e^{-\frac{\g_i^2}{2}}d\g_i
=\frac{\kappa e^{-\frac{\kappa^2}{2}}}{\sqrt{2\pi}}+\frac{(\kappa^2+1)\mbox{erfc}\left ( -\frac{\kappa}{\sqrt{2}}\right )}{2}.\label{eq:fgarlemmaunncorlb}
\end{equation}
Let $\xi_n^{(l)}$ and $\xi_n^{(u)}$ be scalars such that
\begin{eqnarray}
(1-\epsilon_{1}^{(m)})\sqrt{\alpha f_{gar}(\kappa)}-(1+\epsilon_{1}^{(n)})-\epsilon_{5}^{(g)} & > & \frac{\xi_n^{(l)}}{\sqrt{n}}\nonumber \\
(1+\epsilon_{1}^{(m)})\sqrt{\alpha f_{gar}(\kappa)}-(1-\epsilon_{1}^{(n)})+\epsilon_{5}^{(g)} & < & \frac{\xi_n^{(u)}}{\sqrt{n}}.\label{eq:condxinthmstoc30}
\end{eqnarray}
If $\kappa\geq 0$ then
\begin{equation}
 \lim_{n\rightarrow\infty}P(\xi_n^{(l)}\leq \xi_n\leq \xi_n^{(u)})=\lim_{n\rightarrow\infty}P(\min_{\|\x\|_2=1}\max_{\|\lambda\|_2=1,\lambda_i\geq 0}(\xi_n^{(l)}\leq \kappa\lambda^T\1-\lambda^TH\x)\leq \xi_n^{(u)})\geq 1. \label{eq:probthmstoc30poskappa}
\end{equation}
Moreover, if $\kappa< 0$ then
\begin{equation}
 \lim_{n\rightarrow\infty}P(\xi_n\geq \xi_n^{(l)})=\lim_{n\rightarrow\infty}P(\min_{\|\x\|_2=1}\max_{\|\lambda\|_2=1,\lambda_i\geq 0}( \kappa\lambda^T\1-\lambda^TH\x)\geq \xi_n^{(u)})\geq 1. \label{eq:probthmstoc30negkappa}
\end{equation}
\label{thm:Stoc30}
\end{theorem}
\begin{proof}
Presented in \cite{StojnicGardGen13}.
\end{proof}
In a more informal language (essentially ignoring all technicalities and $\epsilon$'s) one has that as long as
\begin{equation}
\alpha>\frac{1}{f_{gar}(\kappa)},\label{eq:condalphauncorlb}
\end{equation}
the problem in (\ref{eq:defprobucor1}) will be infeasible with overwhelming probability. On the other hand, one has that when $\kappa\geq 0$ as long as
\begin{equation}
\alpha<\frac{1}{f_{gar}(\kappa)},\label{eq:condalphauncorubpos}
\end{equation}
the problem in (\ref{eq:defprobucor1}) will be feasible with overwhelming probability. This of course settles the case $\kappa \geq 0$ completely and essentially establishes the storage capacity as $\alpha_c$ which of course matches the prediction given in the introductory analysis presented in \cite{Gar88} and of course rigorously confirmed by the results of \cite{SchTir02,SchTir03,TalBook}. On the other hand, when $\kappa < 0$ it only shows that the storage capacity with overwhelming probability is not higher than the quantity given in \cite{Gar88}. As mentioned above this confirms Talagrand's conjecture 8.4.4 from \cite{TalBook}. However, it does not settle problem (question) 8.4.2 from \cite{TalBook}.

The results obtained based on the above theorem as well as those obtained based on Theorem \ref{thm:SchTirTalSto} are presented in Figure \ref{fig:alfackappa}. When $\kappa\geq 0$ (i.e. when $\alpha\leq 2$) the curve indicates the exact breaking point between the ``overwhelming" feasibility and infeasibility of (\ref{eq:defprobucor1}). On the other hand, when $\kappa< 0$ (i.e. when $\alpha> 2$) the curve is only an upper bound on the storage capacity, i.e. for any value of the pair $(\alpha,\kappa)$ that is above the curve given in Figure \ref{fig:alfackappa}, (\ref{eq:defprobucor1}) is infeasible with overwhelming probability.

Since the case $\kappa<0$ did not appear as settled based on the above presented results we then in \cite{StojnicGardSphNeg13} attempted to lower the upper bounds given in Theorem \ref{thm:Stoc30}. We created a fairly powerful mechanism that produced the following theorem as a way of characterizing the storage capacity of the negative spherical perceptron.

\begin{theorem}
Let $H$ be an $m\times n$ matrix with i.i.d. standard normal components. Let $n$ be large and let $m=\alpha n$, where $\alpha>0$ is a constant independent of $n$. Let $\kappa<0$ be a scalar constant independent of $n$. Set
\begin{equation}
\widehat{\gamma^{(s)}}=\frac{2c_3^{(s)}+\sqrt{4(c_3^{(s)})^2+16}}{8},\label{eq:gamaliftthm}
\end{equation}
and
\begin{equation}
I_{sph}(c_3^{(s)}) = \widehat{\gamma^{(s)}}-\frac{1}{2c_3^{(s)}}\log(1-\frac{c_3^{(s)}}{2\widehat{\gamma^{(s)}}}).\label{eq:Isphthm}
\end{equation}
Set
\begin{equation}
p  =  1+\frac{c_3^{(s)}}{2\gamma_{per}^{(s)}},
q  =  \frac{c_3^{(s)}\kappa}{2\gamma_{per}^{(s)}},
r  =  \frac{c_3^{(s)}\kappa^2}{4\gamma_{per}^{(s)}},
s  =  -\kappa\sqrt{p}+\frac{q}{\sqrt{p}},
C  =  \frac{exp(\frac{q^2}{2p}-r)}{\sqrt{p}},\label{eq:helpdef}
\end{equation}
and
\begin{equation}
I_{per}^{(1)}(c_3^{(s)},\gamma_{per}^{(s)},\kappa)=\frac{1}{2}erfc(\frac{\kappa}{\sqrt{2}})+\frac{C}{2}(erfc(\frac{s}{\sqrt{2}})).\label{eq:Iper1}
\end{equation}
Further, set
\begin{equation}
I_{per}(c_3^{(s)},\alpha,\kappa)=\max_{\gamma_{per}^{(s)}\geq 0}(\gamma_{per}^{(s)}+\frac{1}{c_3^{(s)}}\log(I_{per}^{(1)}(c_3^{(s)},\gamma_{per}^{(s)},\kappa))).\label{eq:Iperthm}
\end{equation}
If $\alpha$ is such that
\begin{equation}
\min_{c_3^{(s)}\geq 0}(-\frac{c_3^{(s)}}{2}+I_{sph}(c_3^{(s)})+I_{per}(c_3^{(s)},\alpha,\kappa))<0,\label{eq:condliftsphnegthm}
\end{equation}
then (\ref{eq:defprobucor1}) is infeasible with overwhelming probability.
\label{thm:liftnegsphper}
\end{theorem}
\begin{proof}
Presented in \cite{StojnicGardSphNeg13}.
\end{proof}

The results one can obtain for the storage capacity based on the above theorem are presented in Figure \ref{fig:liftsphneg} (as mentioned in \cite{StojnicGardSphNeg13}, due to numerical optimizations involved the results presented in Figure \ref{fig:liftsphneg} should be taken only as an illustration; also as discussed in \cite{StojnicGardSphNeg13} taking $c_3^{(s)}\rightarrow 0$ in Theorem \ref{thm:liftnegsphper} produces the results of Theorem \ref{thm:Stoc30}). Even as such, they indicate that a visible improvement in the values of the storage capacity may be possible, though in a range of values of $\alpha$ substantially larger than $2$ (i.e. in a range of $\kappa$'s somewhat smaller than zero). While at this point this observation may look as unrelated to the problem that we will consider in the following section one should keep it in mind (essentially, a conceptually similar conclusion will be made later on when we study the capacities with limited errors).
\begin{figure}[htb]
\centering
\centerline{\epsfig{figure=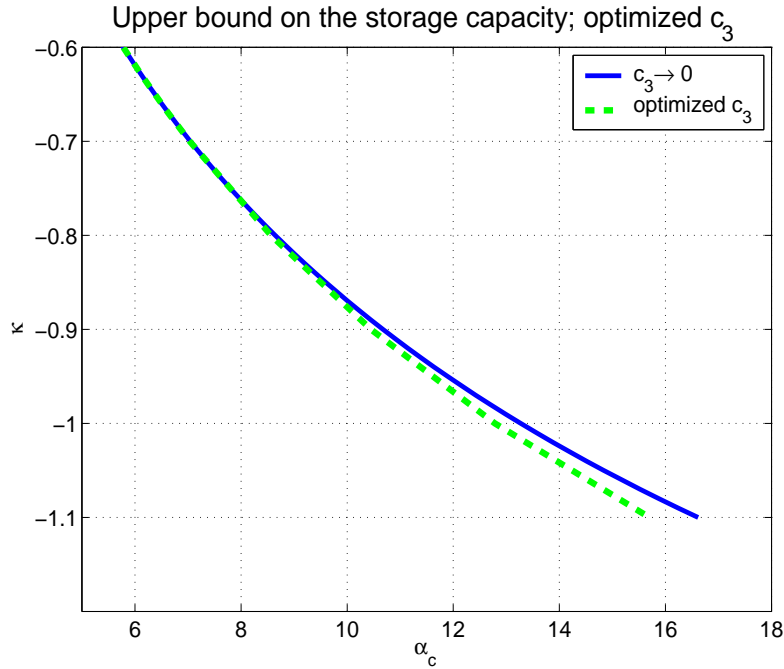,width=10.5cm,height=9cm}}
\caption{$\kappa$ as a function of $\alpha$}
\label{fig:liftsphneg}
\end{figure}

\subsection{Discrete perceptrons}
\label{sec:discper}

Below we present the results/predictions known for the discrete perceptrons. We will mostly focus on the $\pm 1$ perceptron as that one has been studied the most extensively throughout the literature. The known results related to the other two cases that we will study here, namely $0/1$ and box-constrained perceptrons, we find it easier to discuss in parallel as we present our own (these results are a bit involved and we believe that it would be easier to discuss them once we have a few other technical details setup).

Before presenting the concrete known results in this direction we will recall on the problem given in (\ref{eq:defprobucor}) and (\ref{eq:defprobucor1}) and how it changes as one moves from the spherical to $\pm 1$ constraint. Following what was done in Section \ref{sec:rigresnegsphper} one can ask how large $\alpha$ can be so that the following optimization problem is feasible with overwhelming probability
\begin{eqnarray}
& & H\x\geq \kappa\nonumber \\
& & \x_i^2=1,1\leq i\leq n.\label{eq:defprobucor1disc}
\end{eqnarray}
We do, of course, recall that the dimension of $H$ is again $m\times n$ and that $m=\alpha n$ where $\alpha$ is a constant independent of $n$.

As was the case in the previous subsection, we should again preface this brief presentation of the known results by mentioning that a way more is known than what we will present below. However, we will restrict ourselves to the facts that we deem are of most use for the presentation that will follow in later sections.

\subsubsection{Statistical mechanics}
\label{sec:statmech}

As far as a statistical mechanics approach to $\pm 1$ perceptron goes their analytical characterization to a degree have already been started with \cite{Gar88}. Although the main (or the more successful) concern of \cite{GarDer88} was the spherical perceptron, it was also observed that $\pm 1$ perceptron can be handled through the replica mechanisms introduced therein. In a nutshell, what was shown in \cite{GarDer88} (and later also observed in \cite{GutSte90}) related to $\pm 1$ perceptron was the following: assuming that $H_{ij},1\leq i\leq m,1\leq j\leq n$, are i.i.d. symmetric Bernoulli random variables then if $\alpha$ is such that
\begin{equation}
\alpha_{c}(\kappa) =\frac{2}{\pi}(\frac{1}{\sqrt{2\pi}}\int_{-\kappa}^{\infty}(z+\kappa)^2e^{-\frac{z^2}{2}}dz)^{-1},\label{eq:garstorcapdisc}
\end{equation}
then (\ref{eq:defdynfpstr}) holds with overwhelming probability with the restriction on $X_{ji}$ being $X_{ji}\in\{-1,1\}$. Stated in other words (possibly in a more convenient way), if $\alpha$ is such that (\ref{eq:garstorcapdisc}) holds then (\ref{eq:defprobucor1disc}) is feasible with overwhelming probability.

Based on the above characterization one then has that $\alpha_c$ achieves its maximum over positive $\kappa$'s as $\kappa\rightarrow 0$. One in fact easily then has
\begin{equation}
\lim_{\kappa\rightarrow 0}\alpha_c(\kappa)=\frac{4}{\pi}.\label{eq:garstorcapk0disc}
\end{equation}
Of course, it was immediately pointed out already in \cite{GarDer88} that the above $\frac{4}{\pi}$ prediction is essentially not sustainable. In fact not only was it pointed out because of potential instability of the replica approach used in \cite{Gar88}, it was actually rigorously argued through simple combinatorial arguments that $\lim_{\kappa\rightarrow 0}\alpha_c(\kappa)\leq 1$. Many other problems remained open. For example, while it was obvious already based on the considerations presented in \cite{GarDer88} that the storage capacity prediction of $\frac{4}{\pi}$ for the $\kappa=0$ case is an upper bound, it was not clear if one can make such a safe prediction for the entire range of the parameter $\kappa$.

Of course the above considerations then left the replica treatment presented in \cite{Gar88} a bit powerless when it comes to the $\pm 1$ scenario (at the very least in a special case of the so-called zero-thresholds, i.e. when $\kappa=0$). However, many other great works in this direction followed attempting to resolve the problem. A couple of them relied on the statistical mechanics approach mentioned above as well. Of course, as one may expect (and as already had been hinted in \cite{GarDer88}), the first next natural extension of the approach presented in \cite{Gar88,GarDer88} would have been to start breaking replica symmetry. A study in this direction was presented in \cite{KraMez89}. However, as such studies typically may run into substantial numerical problems, the authors in \cite{KraMez89} resorted to a clever way of predicting the critical value for the storage capacity by taking the value where the entropy becomes zero. For $\kappa=0$, that gave an estimate of $\approx 0.83$, substantially lower than $1$, what the above mentioned simple combinatorial bound gives. Similar argument was repeated in \cite{GutSte90} for $\pm 1$ perceptron and extended to $0/1$ perceptron and a few other discrete perceptrons studied therein.

\subsubsection{Rigorous results -- $\pm 1$ perceptron}
\label{sec:rigresdisc}

As far as the rigorous results go we should mention that not much seems to be known. While that does not necessarily mean that the problem is hard, it may imply that it is not super easy either. Among the very first rigorous results are probably those from \cite{KimRoc98}. Roughly speaking, in \cite{KimRoc98}, the authors showed that if $0.005\leq \alpha\leq 0.9937$ then (\ref{eq:defprobucor1disc}) is feasible with overwhelming probability. While these bounds can be improved, improving them to reach anywhere close to $\approx 0.83$ prediction of \cite{KraMez89,GutSte90} does not seem super easy.

We should also mention a seemingly unrelated line of work of Talagrand. Namely, Talagrand studied a variant of the above problem through a more general partition function type of approach, see e.g. \cite{TalBook}. While he was able to show that replica symmetry type of approach would produce rigorous results for such a consideration, he was able to do so in the so-called high-temperature regime. However the problem that he considers boils down to the one of interest here exactly in the opposite, low-temperature regime.

\subsubsection{Simple combinatorial bound -- $\pm 1$ perceptron}
\label{sec:rigresdiscsimpcomb}

Since we have mentioned it in on a couple of occasions in the above discussion we find it useful to also present the simple approach one can use to upper bound the storage capacity of many discrete perceptrons (and certainly of the $\pm 1$ that we consider here). While these bounds may not have been explicitly presented in \cite{GarDer88} the approach that we present below follows the same strategy and we frame it as known result. Namely, one starts by looking at how likely is that each of the inequalities in (\ref{eq:defprobucor1disc}) is satisfied. A simple consideration then gives
\begin{equation}
P(H_{i,:}\x\geq \kappa|\x)=P(g\geq \kappa)=\frac{1}{2}\mbox{erfc}(\frac{\kappa}{\sqrt{2}}),1\leq i\leq m.\label{eq:simpcomb1}
\end{equation}
After accounting for all the inequalities in (\ref{eq:defprobucor1disc}) (essentially all the rows of $H$) one then further has
\begin{equation}
P(H\x\geq \kappa|\x)=(P(H_{i,:}\x\geq \kappa|\x))^m.\label{eq:simpcomb2}
\end{equation}
Using the union bound over all $\x$ then gives
\begin{equation}
P(\exists \x|H\x\geq \kappa)\leq 2^nP(H\x\geq \kappa|\x)=2^n(P(H_{i,:}\x\geq \kappa|\x))^m.\label{eq:simpcomb3}
\end{equation}
A combination of (\ref{eq:simpcomb1}) and (\ref{eq:simpcomb3}) then gives
\begin{equation}
P(\exists \x|H\x\geq \kappa)\leq 2^n\left (\frac{1}{2}\mbox{erfc}(\frac{\kappa}{\sqrt{2}})\right )^m.\label{eq:simpcomb4}
\end{equation}
From (\ref{eq:simpcomb4}) one then has that if $\alpha=\frac{m}{n}$ is such that
\begin{equation}
\alpha> -\frac{\log(2)}{\log\left (\frac{1}{2}\mbox{erfc}(\frac{\kappa}{\sqrt{2}})\right )},\label{eq:simpcomb5}
\end{equation}
then
\begin{equation}
\lim_{n\rightarrow\infty}P(\exists \x|H\x\geq \kappa)\leq \lim_{n\rightarrow\infty}
2^n\left (\frac{1}{2}\mbox{erfc}(\frac{\kappa}{\sqrt{2}})\right )^m=0.\label{eq:simpcomb6}
\end{equation}
The upper bounds one can obtain on the storage capacity based on the above consideration (in particular based on (\ref{eq:simpcomb5})) are presented in Figure \ref{fig:discsimpcomb}. Of course, these bounds can be improved (as mentioned earlier, one of possible such improvements was already presented in \cite{KimRoc98}). However, here our goal is more to recall on the results that relate to the ones that we will present in this paper rather than on the best possible ones.
\begin{figure}[htb]
\centering
\centerline{\epsfig{figure=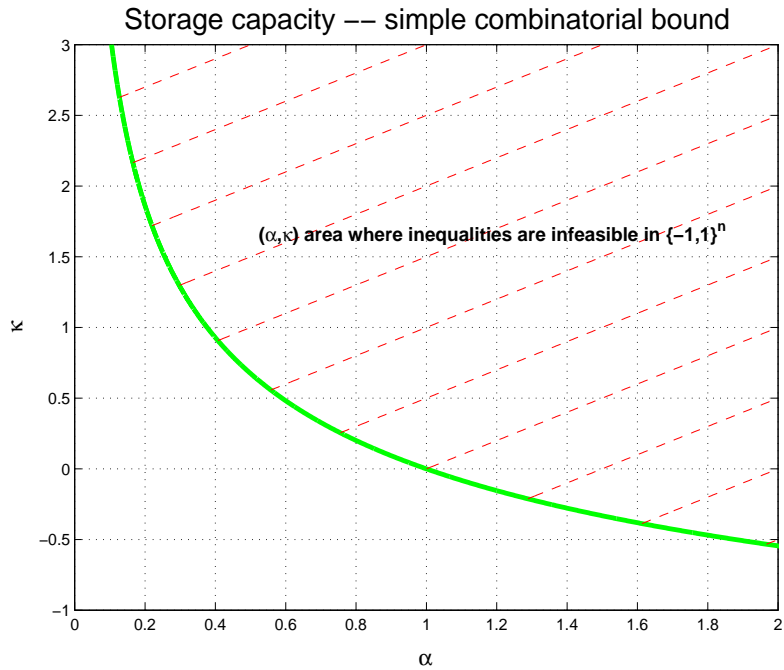,width=10.5cm,height=9cm}}
\caption{$\kappa$ as a function of $\alpha$; simple combinatorial bound; $\x\in\left \{-\frac{1}{\sqrt{n}},\frac{1}{\sqrt{n}}\right \}^n$}
\label{fig:discsimpcomb}
\end{figure}

\section{$\pm 1$ perceptrons}
\label{sec:discper}

In this section we will present a collection of mathematically rigorous results related to $\pm 1$ perceptrons. We will rely on many simplifications of the original perceptron setup from Section \ref{sec:mathsetupper} introduced in \cite{StojnicGardGen13,StojnicGardSphNeg13,StojnicGardSphErr13} and presented in Section \ref{sec:discper}. To that end we start by recalling that for all practical purposes needed here (and those we needed in \cite{StojnicGardGen13,StojnicGardSphNeg13,StojnicGardSphErr13}) the storage capacity of $\pm 1$ perceptron can be considered through the feasibility problem given in (\ref{eq:defprobucor1disc}) which we restate below
\begin{eqnarray}
& & H\x\geq \kappa\nonumber \\
& & \x_i^2=\frac{1}{n},1\leq i\leq n.\label{eq:defprobucor2pm1}
\end{eqnarray}
We recall as well, that as argued in \cite{StojnicGardGen13,StojnicGardSphNeg13,StojnicGardSphErr13} (and as mentioned in the previous section) one can assume that the elements of $H$ are i.i.d. standard normals and that the dimension of $H$ is $m\times n$, where as earlier we keep the linear regime, i.e. continue to assume that $m=\alpha n$ where $\alpha$ is a constant independent of $n$. Now, if all inequalities in (\ref{eq:defprobucor2pm1}) are satisfied one can have that the dynamics established will be stable and all $m$ patterns could be successfully stored. Following the strategy presented in \cite{StojnicGardGen13,StojnicGardSphNeg13,StojnicGardSphErr13} (and briefly recalled on in Section \ref{sec:rigresnegsphper} one can then reformulate (\ref{eq:defprobucor2pm1}) so that the feasibility problem of interest becomes
\begin{eqnarray}
\xi_{\pm 1}=\min_{\x} \max_{\lambda\geq 0} & &  \kappa\lambda^T\1- \lambda^TH\x \nonumber \\
\mbox{subject to} & & \|\lambda\|_2= 1\nonumber \\
& & \x_i^2=\frac{1}{n},1\leq i\leq n.\label{eq:feaspm1}
\end{eqnarray}
Clearly, following the logic we presented in Section \ref{sec:rigresnegsphper}, the sign of $\xi_{\pm 1}$ determines the feasibility of (\ref{eq:defprobucor2pm1}). In particular, if $\xi_{\pm 1}>0$ then (\ref{eq:defprobucor2pm1}) is infeasible. Given the random structure of the problem (we recall that $H$ is random) one can then pose the following probabilistic feasibility question: how small can $m$ be so that $\xi_{\pm 1}$ in (\ref{eq:feaspm1}) is positive and (\ref{eq:defprobucor2pm1}) is infeasible with overwhelming probability? In what follows we will attempt to provide an answer to such a question.

\subsection{Probabilistic analysis}
\label{sec:probanalrigpm1}

In this section we will present a probabilistic analysis of the above optimization problem given in (\ref{eq:feaspm1}). In a nutshell, we will provide a relation between $\kappa$ and $\alpha=\frac{m}{n}$ so that with overwhelming probability over $H$ $\xi_{\pm 1}>0$. This will, of course, based on the above discussion then be enough to conclude that the problem in (\ref{eq:defprobucor2pm1}) is infeasible with overwhelming probability when $\kappa$ and $\alpha=\frac{m}{n}$ satisfy such a relation.

The analysis that we will present below will to a degree rely on a strategy we developed in \cite{StojnicRegRndDlt10,StojnicGorEx10} and utilized in \cite{StojnicGardGen13} when studying the storage capacity of the standard spherical perceptrons. We start by recalling on a set of probabilistic results from \cite{Gordon85,Gordon88} that were used as an integral part of the strategy developed in \cite{StojnicRegRndDlt10,StojnicGorEx10,StojnicGardGen13}.
\begin{theorem}(\cite{Gordon88,Gordon85})
\label{thm:Gordonmesh1} Let $X_{ij}$ and $Y_{ij}$, $1\leq i\leq n,1\leq j\leq m$, be two centered Gaussian processes which satisfy the following inequalities for all choices of indices
\begin{enumerate}
\item $E(X_{ij}^2)=E(Y_{ij}^2)$
\item $E(X_{ij}X_{ik})\geq E(Y_{ij}Y_{ik})$
\item $E(X_{ij}X_{lk})\leq E(Y_{ij}Y_{lk}), i\neq l$.
\end{enumerate}
Then
\begin{equation*}
P(\bigcap_{i}\bigcup_{j}(X_{ij}\geq \lambda_{ij}))\leq P(\bigcap_{i}\bigcup_{j}(Y_{ij}\geq \lambda_{ij})).
\end{equation*}
\end{theorem}
The following, more simpler, version of the above theorem relates to the expected values.
\begin{theorem}(\cite{Gordon85,Gordon88})
\label{thm:Gordonmesh2} Let $X_{ij}$ and $Y_{ij}$, $1\leq i\leq n,1\leq j\leq m$, be two centered Gaussian processes which satisfy the following inequalities for all choices of indices
\begin{enumerate}
\item $E(X_{ij}^2)=E(Y_{ij}^2)$
\item $E(X_{ij}X_{ik})\geq E(Y_{ij}Y_{ik})$
\item $E(X_{ij}X_{lk})\leq E(Y_{ij}Y_{lk}), i\neq l$.
\end{enumerate}
Then
\begin{equation*}
E(\min_{i}\max_{j}(X_{ij}))\leq E(\min_i\max_j(Y_{ij})).
\end{equation*}
\end{theorem}

Now, since all random quantities of interest below will concentrate around its mean values it will be enough to study only their averages. However, since it will not make writing of what we intend to present in the remaining parts of this section substantially more complicated we will present a complete probabilistic treatment and will leave the studying of the expected values for the presentation that we will give in the following subsection where such a consideration will substantially simplify the exposition.

We will make use of Theorem \ref{thm:Gordonmesh1} through the following lemma (the lemma is an easy consequence of Theorem \ref{thm:Gordonmesh1} and in fact is fairly similar to Lemma 3.1 in \cite{Gordon88}, see also \cite{StojnicHopBnds10,StojnicGardGen13} for similar considerations).
\begin{lemma}
Let $H$ be an $m\times n$ matrix with i.i.d. standard normal components. Let $\g$ and $\h$ be $m\times 1$ and $n\times 1$ vectors, respectively, with i.i.d. standard normal components. Also, let $g$ be a standard normal random variable and let $\zeta_{\lambda}$ be a function of $\x$. Then
\begin{equation}
P(\min_{\x_i^2=\frac{1}{n}}\max_{\|\lambda\|_2=1,\lambda_i\geq 0}(-\lambda^TH\x+g-\zeta_{\lambda})\geq 0)\\\geq
P(\min_{\x_i^2=\frac{1}{n}}\max_{\|\lambda\|_2=1,\lambda_i\geq 0}(\g^T\lambda+\h^T\x-\zeta_{\lambda})\geq 0).\label{eq:negproblemma}
\end{equation}\label{lemma:negproblemma}
\end{lemma}
\begin{proof}
The proof is basically similar to the proof of Lemma 3.1 in \cite{Gordon88} as well as to the proof of Lemma 7 in \cite{StojnicHopBnds10}. The only difference is in allowed sets of values for $\x$ and $\lambda$. Such a difference introduces no structural changes in the proof though.
\end{proof}

Let $\zeta_{\lambda}=-\kappa\lambda^T\1+\epsilon_{5}^{(g)}\sqrt{n}+\xi_{\pm 1}^{(l)}$ with $\epsilon_{5}^{(g)}>0$ being an arbitrarily small constant independent of $n$. We will first look at the right-hand side of the inequality in (\ref{eq:negproblemma}). The following is then the probability of interest
\begin{equation}
P(\min_{\x_i^2=1}\max_{\|\lambda\|_2=1,\lambda_i\geq 0}(\g^T\lambda+\h^T\x+\kappa\lambda^T\1-\epsilon_{5}^{(g)}\sqrt{n})\geq \xi_{\pm 1}^{(l)}).\label{eq:negprobanal0}
\end{equation}
After solving the minimization over $\x$ one obtains
\begin{equation}
\hspace{-.3in}P(\min_{\x_i^2=1}\max_{\|\lambda\|_2=1,\lambda_i\geq 0}(\g^T\lambda+\h^T\x+\kappa\lambda^T\1-\epsilon_{5}^{(g)}\sqrt{n})\geq \xi_{\pm 1}^{(l)})=P(\|(\g+\kappa\1)_+\|_2-\sum_{i=1}^{n}|\h_i|-\epsilon_{5}^{(g)}\sqrt{n}\geq \xi_{\pm 1}^{(l)}),\label{eq:negprobanal1}
\end{equation}
where $(\g+\kappa\1)_+$ is $(\g+\kappa\1)$ vector with negative components replaced by zeros. Following line by line what was done in \cite{StojnicGardGen13} after equation $(13)$ one then has
\begin{multline}
P(\min_{\x_i^2=1}\max_{\|\lambda\|_2=1,\lambda_i\geq 0}(\g^T\lambda+\h^T\x+\kappa\lambda^T\1-\epsilon_{5}^{(g)}\sqrt{n})\geq \xi_{\pm 1}^{(l)})\\\geq
(1-e^{-\epsilon_{2}^{(m)} m})(1-e^{-\epsilon_{2}^{(n)} n})
P((1-\epsilon_{1}^{(m)})\sqrt{mf_{gar}(\kappa)}-(1+\epsilon_{1}^{(n)})\sqrt{n}-\epsilon_{5}^{(g)}\sqrt{n}\geq \xi_{\pm 1}^{(l)}).
\label{eq:negprobanal2}
\end{multline}
where
\begin{equation}
f_{gar}(\kappa)=\frac{1}{\sqrt{2\pi}}\int_{-\kappa}^{\infty}(\g_i+\kappa)^2e^{-\frac{\g_i^2}{2}}d\g_i,\label{eq:fgar}
\end{equation}
$\epsilon_{5}^{(g)}$, $\epsilon_1^{(m)}$, and $\epsilon_1^{(n)}$ are arbitrarily small positive constants and $\epsilon_2^{(m)}$ and $\epsilon_2^{(n)}$ are constants possibly dependent on $\epsilon_1^{(m)}$, $f_{gar}(\kappa)$, and $\epsilon_1^{(n)}$, respectively but independent of $n$.
If
\begin{eqnarray}
& & (1-\epsilon_{1}^{(m)})\sqrt{mf_{gar}(\kappa)}-(1+\epsilon_{1}^{(n)})\sqrt{\frac{2}{\pi}}\sqrt{n}-\epsilon_{5}^{(g)}\sqrt{n}>\xi_{\pm 1}^{(l)}\nonumber \\
& \Leftrightarrow & (1-\epsilon_{1}^{(m)})\sqrt{\alpha f_{gar}(\kappa)}-(1+\epsilon_{1}^{(n)})\sqrt{\frac{2}{\pi}}-\epsilon_{5}^{(g)}>\frac{\xi_{\pm 1}^{(l)}}{\sqrt{n}},\label{eq:negcondxipu}
\end{eqnarray}
one then has from (\ref{eq:negprobanal2})
\begin{equation}
\lim_{n\rightarrow\infty}P(\min_{\x_i^2=1}\max_{\|\lambda\|_2=1,\lambda_i\geq 0}(\g^T\lambda+\h^T\x+\kappa\lambda^T\1-\epsilon_{5}^{(g)}\sqrt{n})\geq \xi_{\pm 1}^{(l)})\geq 1.\label{eq:negprobanal3}
\end{equation}

We will also need the following simple estimate related to the left hand side of the inequality in (\ref{eq:negproblemma}). From (\ref{eq:negproblemma}) one has the following as the probability of interest
\begin{equation}
P(\min_{\x_i^2=1}\max_{\|\lambda\|_2=1,\lambda_i\geq 0}(\kappa\lambda^T\1-\lambda^TH\x+g-\epsilon_{5}^{(g)}\sqrt{n}-\xi_{\pm 1}^{(l)})\geq 0).\label{eq:leftnegprobanal0}
\end{equation}
Following again what was done in \cite{StojnicGardGen13} between equations $(21)$ and $(24)$ one has,
assuming that (\ref{eq:negcondxipu}) holds,
\begin{equation}
\hspace{-.5in}\lim_{n\rightarrow\infty}P(\min_{\x_i^2=1}\max_{\|\lambda\|_2=1,\lambda_i\geq 0}(\kappa\lambda^T\1-\lambda^TH\x)\geq \xi_{\pm1}^{(l)})\geq \lim_{n\rightarrow\infty}P(\min_{\x_i^2=1}\max_{\|\lambda\|_2=1,\lambda_i\geq 0}(\g^T\y+\h^T\x+\kappa\lambda^T\1-\epsilon_{5}^{(g)}\sqrt{n})\geq \xi_{\pm 1}^{(l)})\geq 1.\label{eq:leftnegprobanal3}
\end{equation}

We summarize the above results in the following theorem.

\begin{theorem}
Let $H$ be an $m\times n$ matrix with i.i.d. standard normal components. Let $n$ be large and let $m=\alpha n$, where $\alpha>0$ is a constant independent of $n$. Let $\xi_{\pm 1}$ be as in (\ref{eq:feaspm1}) and let $\kappa$ be a scalar constant independent of $n$. Let all $\epsilon$'s be arbitrarily small constants independent of $n$. Further, let $\g_i$ be a standard normal random variable and set
\begin{equation}
f_{gar}(\kappa) = \frac{1}{\sqrt{2\pi}}\int_{-\kappa}^{\infty}(\g_i+\kappa)^2e^{-\frac{\g_i^2}{2}}d\g_i
=\frac{\kappa e^{-\frac{\kappa^2}{2}}}{\sqrt{2\pi}}+\frac{(\kappa^2+1)\mbox{erfc}\left ( -\frac{\kappa}{\sqrt{2}}\right )}{2}.\label{eq:thmerrc30}
\end{equation}

Let $\xi_{\pm 1}^{(l)}$ be a scalar such that
\begin{equation}
(1-\epsilon_{1}^{(m)})\sqrt{\alpha f_{gar}(\kappa)}-(1+\epsilon_{1}^{(n)})\sqrt{\frac{2}{\pi}}-\epsilon_{5}^{(g)}  >  \frac{\xi_{\pm 1}^{(l)}}{\sqrt{n}}.\label{eq:condxinthmstoc30}
\end{equation}
Then
\begin{equation}
\hspace{-.3in} \lim_{n\rightarrow\infty}P(\xi_{\pm 1}\geq \xi_{\pm 1}^{(l)})=\lim_{n\rightarrow\infty}P(\min_{\x_i^2=1}\max_{\|\lambda\|_2=1,\lambda_i\geq 0}(\kappa\lambda^T\1-\lambda^TH\x)\geq \xi_{\pm 1}^{(l)})\geq 1. \label{eq:probthmstoc30poskappa}
\end{equation}
\label{thm:pm1c30}
\end{theorem}
\begin{proof}
Follows from the above discussion and the analysis presented in \cite{StojnicGardGen13}.
\end{proof}

In a more informal language (as earlier, essentially ignoring all technicalities and $\epsilon$'s) one has that as long as
\begin{equation}
\alpha>\frac{2}{\pi}\frac{1}{f_{gar}(\kappa)},\label{eq:condalphauncorlberr}
\end{equation}
the problem in (\ref{eq:defprobucor2pm1}) will be infeasible with overwhelming probability. It is an easy exercise to show that the right hand side of (\ref{eq:condalphauncorlberr}) matches the right-hand side of (\ref{eq:garstorcapdisc}). This is then enough to conclude that the prediction for the storage capacity given in \cite{GarDer88} for $\pm 1$ perceptron is in fact a rigorous upper bound on its true value.

The results obtained based on the above theorem as well as those predicted based on the replica theory and given in (\ref{eq:garstorcapdisc}) (and of course in \cite{GarDer88}) are presented in Figure \ref{fig:discperc30}. For the values of $\alpha$ that are to the right of the given curve the memory will not operate correctly with overwhelming probability. This of course follows from the fact that with overwhelming probability over $H$ the inequalities in (\ref{eq:defprobucor2pm1}) will not be simultaneously satisfiable.

\begin{figure}[htb]
\centering
\centerline{\epsfig{figure=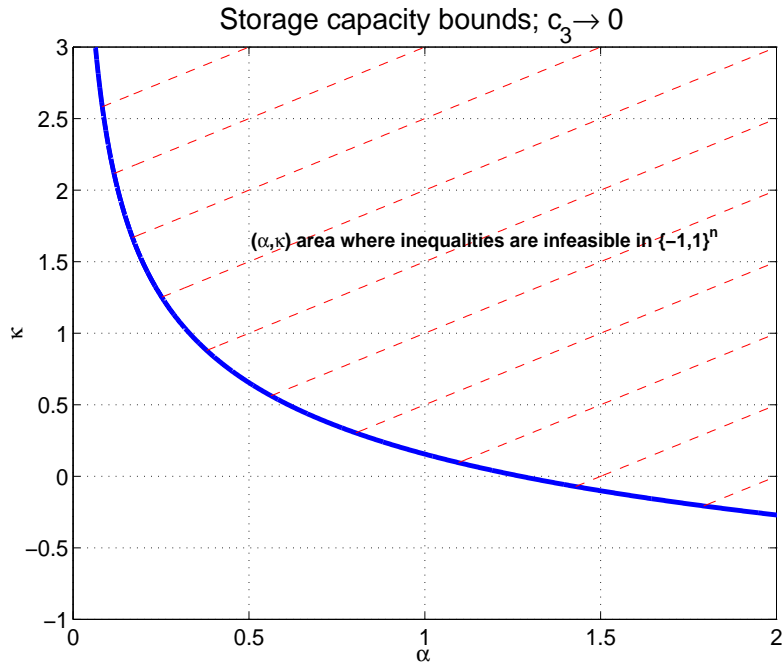,width=10.5cm,height=9cm}}
\caption{$\kappa$ as a function of $\alpha$; $\x\in\left \{-\frac{1}{\sqrt{n}},\frac{1}{\sqrt{n}}\right \}^n$}
\label{fig:discperc30}
\end{figure}

\subsection{Lowering the storage capacity upper bound}
\label{sec:discperlow}

The results we presented in the previous section provide a rigorous upper bound on the storage capacity of $\pm 1$ perceptron. As we have mentioned in Section \ref{sec:knownres} it had been known already from the initial considerations in \cite{GarDer88} that the upper bounds we presented in the previous sections for certain values of $\kappa$ are strict (and in fact quite far away from the optimal values). In this section we will follow the strategy we employed in \cite{StojnicGardSphNeg13,StojnicGardSphErr13} for studying scenarios where the standard upper bounds are potentially non-exact. Such a strategy essentially attempts to lower the upper bounds provided in the previous subsection. It does so by attempting to lift the lower bounds on $\xi_{\pm 1}$. After doing so we will come to a point to reveal an interesting phenomenon happening in the analysis of $\pm 1$ perceptrons. Namely, in certain range of $\kappa$ the upper bounds of the previous sections will indeed end up being lowered by the strategy that we will present. However, it will turn out that the only lowering that we were able to uncover is the one that corresponds to the simple combinatorial bounds given in Section \ref{sec:rigresdiscsimpcomb}. However, before arriving to such a conclusion we will need to resolve a few technical problems.

Before proceeding further with the presentation of the above mentioned strategy, we first recall on a few technical details from previous sections that we will need here again. We start by recalling on the optimization problem that we will consider here. As is probably obvious, it is basically the one given in (\ref{eq:feaspm1})
\begin{eqnarray}
\xi_{\pm 1}=\min_{\x} \max_{\lambda\geq 0} & &  \kappa\lambda^T\1- \lambda^T H\x \nonumber \\
\mbox{subject to} & & \|\lambda\|_2= 1\nonumber \\
& & \x_i^2=1.\label{eq:feaspm1low}
\end{eqnarray}
As mentioned below (\ref{eq:feaspm1}), a probabilistic characterization of the sign of $\xi_{\pm 1}$ would be enough to determine the storage capacity or its bounds. Below, we provide a way similar to the one from the previous subsection that can also be used to probabilistically characterize $\xi_{\pm 1}$. Moreover, as mentioned at the beginning of the previous section, since $\xi_{\pm 1}$ will concentrate around its mean for our purposes here it will then be enough to study only its mean $E\xi_{\pm 1}$. We do so by relying on the strategy developed in \cite{StojnicMoreSophHopBnds10} (and employed in \cite{StojnicGardSphNeg13,StojnicGardSphErr13}) and ultimately on the following set of results from \cite{Gordon85}. (The following theorem presented in \cite{StojnicMoreSophHopBnds10} is in fact a slight alternation of the original results from \cite{Gordon85}.)
\begin{theorem}(\cite{Gordon85})
\label{thm:Gordonneg1} Let $X_{ij}$ and $Y_{ij}$, $1\leq i\leq n,1\leq j\leq m$, be two centered Gaussian processes which satisfy the following inequalities for all choices of indices
\begin{enumerate}
\item $E(X_{ij}^2)=E(Y_{ij}^2)$
\item $E(X_{ij}X_{ik})\geq E(Y_{ij}Y_{ik})$
\item $E(X_{ij}X_{lk})\leq E(Y_{ij}Y_{lk}), i\neq l$.
\end{enumerate}
Let $\psi_{ij}()$ be increasing functions on the real axis. Then
\begin{equation*}
E(\min_{i}\max_{j}\psi_{ij}(X_{ij}))\leq E(\min_{i}\max_{j}\psi_{ij}(Y_{ij})).
\end{equation*}
Moreover, let $\psi_{ij}()$ be decreasing functions on the real axis. Then
\begin{equation*}
E(\max_{i}\min_{j}\psi_{ij}(X_{ij}))\geq E(\max_{i}\min_{j}\psi_{ij}(Y_{ij})).
\end{equation*}
\begin{proof}
The proof of all statements but the last one is of course given in \cite{Gordon85}. The proof of the last statement trivially follows and in a slightly different scenario is given for completeness in \cite{StojnicMoreSophHopBnds10}.
\end{proof}
\end{theorem}
The strategy that we will present below will utilize the above theorem to lift the above mentioned lower bound on $\xi_{\pm 1}$ (of course since we talk in probabilistic terms, under bound on $\xi_{\pm 1}$ we essentially assume a bound on $E\xi_{\pm 1}$). We do mention again that in Section \ref{sec:probanalrigpm1} we relied on a variant of the above theorem to create a probabilistic lower bound on $\xi_{\pm 1}$. However, the strategy employed in Section \ref{sec:probanalrigpm1} relied only on a basic version of the above theorem which assumes $\psi_{ij}(x)=x$. Here, we will substantially upgrade the strategy from Section \ref{sec:probanalrigpm1} by looking at a very simple (but way better) different version of $\psi_{ij}()$.

\subsubsection{Lifting lower bound on $\xi_{\pm 1}$}
\label{sec:uncorgardlb}

In \cite{StojnicMoreSophHopBnds10,StojnicGardSphNeg13} we established lemmas very similar to the following one:
\begin{lemma}
Let $A$ be an $m\times n$ matrix with i.i.d. standard normal components. Let $\g$ and $\h$ be $m\times 1$ and $n\times 1$ vectors, respectively, with i.i.d. standard normal components. Also, let $g$ be a standard normal random variable and let $c_3$ be a positive constant. Then
\begin{equation}
E(\max_{\x_i^2=1}\min_{\|\lambda\|_2=1,\lambda_i\geq 0}e^{-c_3(-\lambda^T H\x + g +\kappa\lambda^T\1)})\leq E(\max_{\x_i^2=1}\min_{\|\lambda\|_2=1,\lambda_1\geq 0}e^{-c_3(\g^T\lambda+\h^T\x+\kappa\lambda^T\1)}).\label{eq:negexplemmalow}
\end{equation}\label{lemma:negexplemmalow}
\end{lemma}
\begin{proof}
The proof is the same as the proof of the corresponding lemma in \cite{StojnicMoreSophHopBnds10}. The only difference is in the structure of the sets of allowed values for $\x$ and $\lambda$. However, such a difference introduces no structural changes in the proof.
\end{proof}

Following step by step what was done after Lemma 3 in \cite{StojnicMoreSophHopBnds10} one arrives at the following analogue of \cite{StojnicMoreSophHopBnds10}'s equation $(57)$:
\begin{multline}
E(\min_{\x_i^2=1}\max_{\|\lambda\|_2=1,\lambda_i\geq 0}(-\lambda^TH\x+\kappa\lambda^T\1))\\\hspace{-.3in}\geq
\frac{c_3}{2}-\frac{1}{c_3}\log(E(\max_{\x_i^2=1}(e^{-c_3\h^T\x})))
-\frac{1}{c_3}\log(E(\min_{\|\lambda\|_2=1,\lambda_i\geq 0}(e^{-c_3(\g^T\lambda+\kappa\lambda^T\1)}))).\\\label{eq:chneg8}
\end{multline}
Let $c_3=c_3^{(s)}\sqrt{n}$ where $c_3^{(s)}$ is a constant independent of $n$. Then (\ref{eq:chneg8}) becomes
\begin{equation}
\frac{E(\min_{\x_i^2=1}\max_{\|\lambda\|_2=1,\lambda_i\geq 0}(-\lambda^T H\x+\kappa\lambda^T\1))}{\sqrt{n}}
\geq
 -(-\frac{c_3^{(s)}}{2}+I_{sph}(c_3^{(s)})+I_{\pm 1}(c_3^{(s)},\alpha,\kappa)),\label{eq:chneg9}
\end{equation}
where
\begin{eqnarray}
I_{\pm 1}(c_3^{(s)}) & = & \frac{1}{nc_3^{(s)}}\log(E(\max_{\x_i^2=1}(e^{-c_3^{(s)}\sqrt{n}\h^T\x})))\nonumber \\
I_{per}(c_3^{(s)},\alpha,\kappa) & = & \frac{1}{nc_3^{(s)}}\log(E(\min_{\|\lambda\|_2=1,\lambda_i\geq 0}(e^{-c_3^{(s)}\sqrt{n}(\g^T\lambda+\kappa\lambda^T\1)}))).\nonumber \\\label{eq:defIs}
\end{eqnarray}
Moreover, \cite{StojnicMoreSophHopBnds10} also established
\begin{equation}
I_{\pm 1}(c_3^{(s)}) = \frac{1}{nc_3^{(s)}}\log(E(\max_{\x_i^2=1}(e^{-c_3^{(s)}\sqrt{n}\h^T\x})))
 = \frac{c_3^{(s)}}{2}+\frac{1}{c_3^{(s)}}\log(\mbox{erfc}(-\frac{c_3^{(s)}}{\sqrt{2}})).\label{eq:ubmorsoph}
\end{equation}
Furthermore, \cite{StojnicGardSphNeg13} established a way to determine $I_{per}(c_3^{(s)},\alpha,\kappa)$. It is exactly as specified in Theorem \ref{thm:liftnegsphper}.

We summarize the above observations in the following theorem.

\begin{theorem}
Let $H$ be an $m\times n$ matrix with i.i.d. standard normal components. Let $n$ be large and let $m=\alpha n$, where $\alpha>0$ is a constant independent of $n$. Let $\xi_{\pm 1}$ be as in (\ref{eq:feaspm1}) and let $\kappa$ be a scalar constant independent of $n$. Set
\begin{equation}
I_{\pm 1}(c_3^{(s)})
 = \frac{c_3^{(s)}}{2}+\frac{1}{c_3^{(s)}}\log(\mbox{erfc}(-\frac{c_3^{(s)}}{\sqrt{2}})).\label{eq:Ipm1thmlow}
\end{equation}
and
\begin{equation}
p  =  1+\frac{c_3^{(s)}}{2\gamma_{per}^{(s)}},
q  =  \frac{c_3^{(s)}\kappa}{2\gamma_{per}^{(s)}},
r  =  \frac{c_3^{(s)}\kappa^2}{4\gamma_{per}^{(s)}},
s  =  -\kappa\sqrt{p}+\frac{q}{\sqrt{p}},
C  =  \frac{exp(\frac{q^2}{2p}-r)}{\sqrt{p}}.\label{eq:helpdefpm1thmlow}
\end{equation}
Also, set
\begin{equation}
I_{per}^{(1)}(c_3^{(s)},\gamma_{per}^{(s)},\kappa)=\frac{1}{2}erfc(\frac{\kappa}{\sqrt{2}})+\frac{C}{2}(erfc(\frac{s}{\sqrt{2}})).\label{eq:Iper1pm1thmlow}
\end{equation}
and
\begin{equation}
I_{per}(c_3^{(s)},\alpha,\kappa)=\max_{\gamma_{per}^{(s)}\geq 0}(\gamma_{per}^{(s)}+\frac{1}{c_3^{(s)}}\log(I_{per}^{(1)}(c_3^{(s)},\gamma_{per}^{(s)},\kappa))).\label{eq:Iperpm1thmlow}
\end{equation}
If $\alpha$ is such that
\begin{equation}
\min_{c_3^{(s)}\geq 0}(-\frac{c_3^{(s)}}{2}+I_{\pm 1}(c_3^{(s)})+I_{per}(c_3^{(s)},\alpha,\kappa))<0,\label{eq:condalphapm1thmlow}
\end{equation}
then (\ref{eq:defprobucor2pm1}) is infeasible with overwhelming probability.
\label{thm:pm1optc3}
\end{theorem}
\begin{proof}
Follows from the previous discussion by combining (\ref{eq:feaspm1}) and (\ref{eq:chneg9}), and  by noting that the bound given in (\ref{eq:chneg9}) holds for any $c_3^{(s)}\geq 0$ and could therefore be tightened by additionally optimizing over $c_3^{(s)}\geq 0$.
\end{proof}

The results one can obtain for the storage capacity based on the above theorem are presented in Figure \ref{fig:discperoptc3}. In addition to that we also present the results one can obtain based on Theorem \ref{thm:pm1c30}. These are denoted by $c_3\rightarrow 0$ as they can be obtained from Theorem \ref{thm:pm1optc3} by taking $c_3^{(s)}\rightarrow 0$. Furthermore, we also present the results one can obtain based on the simple combinatorial bound discussed in Section \ref{sec:rigresdiscsimpcomb} and presented in Figure \ref{fig:discsimpcomb}. As can be seen from Figure \ref{fig:discperoptc3} the optimal values that we found for $c_3^{(s)}$ correspond either to $0$ or to a $c_3^{(s)}$ that eventually gives an $\alpha$ that matches the one obtained in Section \ref{sec:rigresdiscsimpcomb}. In fact, when $c_3^{(s)}=0$ is not optimal we only found $c_{3}^{(s)}\rightarrow \infty$ as a better option. A simple analytical transformation of the results presented in the above theorem (assuming $c_{3}^{(s)}\rightarrow \infty$) indeed produces the upper bound given in (\ref{eq:simpcomb5}). We would view this as in a way somewhat surprising result.

Also, we would like to mention that the results presented in Figure \ref{fig:discperoptc3} should be taken only as an illustration. They are obtained as a result of a numerical optimization. Remaining finite precision errors are of course possible and could affect the validity of the obtained results (we do believe though that this is not the case). Either way, we would like to emphasize once again that the results presented in Theorem \ref{thm:pm1optc3} are completely mathematically rigorous. Their representation given in Figure \ref{fig:discperoptc3} may have been a bit imprecise due to numerical computations needed to obtain the plots shown in the figure.

\begin{figure}[htb]
\centering
\centerline{\epsfig{figure=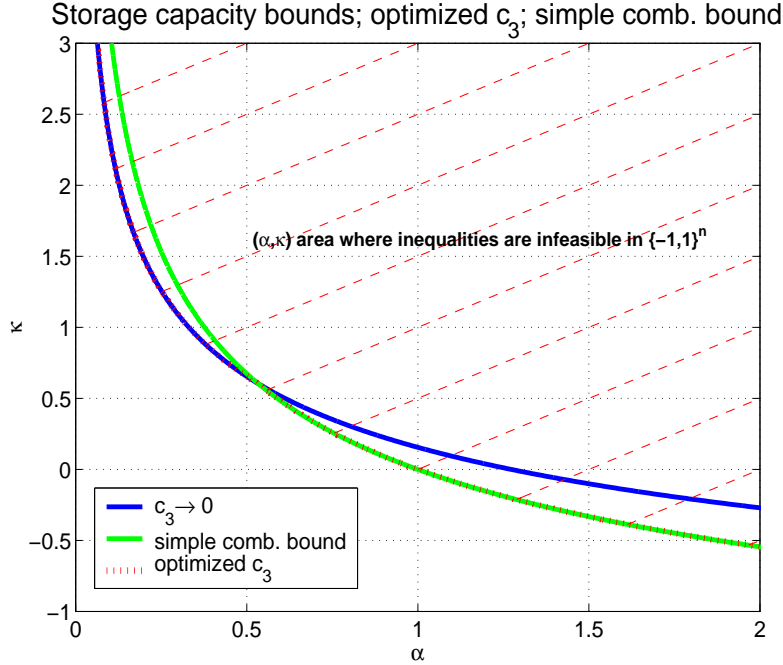,width=10.5cm,height=9cm}}
\caption{$\kappa$ as a function of $\alpha$; optimized $c_3^{(s)}$; $\x\in\left \{-\frac{1}{\sqrt{n}},\frac{1}{\sqrt{n}}\right \}^n$}
\label{fig:discperoptc3}
\end{figure}

\section{$0/1$ perceptrons}
\label{sec:01per}

In this section we will present a collection of mathematically rigorous results related to $0/1$ perceptrons. To make the presentation as smooth as possible we will try to emulate the exposition of Section \ref{sec:discper} as much as possible. As in Section \ref{sec:discper} we will rely on many simplifications of the original perceptron setup introduced in Section \ref{sec:mathsetupper} (and of course earlier in \cite{StojnicGardGen13,StojnicGardSphNeg13,StojnicGardSphErr13}). Following what was done in Section \ref{sec:discper} it is not that hard to recognize that the storage capacity of $0/1$ perceptron can be considered through the following feasibility problem
\begin{eqnarray}
& & H\x\geq \kappa\nonumber \\
& & \x_i\in \left \{0,\frac{1}{\sqrt{n}}\right \},1\leq i\leq n.\label{eq:defprobucor201per}
\end{eqnarray}
As in Section \ref{sec:discper} (and obviously as argued in \cite{StojnicGardGen13,StojnicGardSphNeg13,StojnicGardSphErr13}) one can assume that the elements of $H$ are i.i.d. standard normals and that the dimension of $H$ is $m\times n$. Moreover, we will continue to work in the linear regime, i.e. we will continue to assume that $m=\alpha n$ where $\alpha$ is a constant independent of $n$. Now, if all inequalities in (\ref{eq:defprobucor201per}) are satisfied one can have that the perceptron dynamics discussed in Section \ref{sec:mathsetupper} will be stable and all $m$ patterns could be successfully stored. Before proceeding further with following the exposition of the previous section, we will scale the above problem a bit. In our view, the following transformation will make the presentation of what follows substantially easier on the one hand and will enable us to maintain the same type of scaling as in known references, see e.g. \cite{GutSte90}. Namely, we will first discretize the problem a bit. It is not that hard to see that the set of all allowed $\x$ in (\ref{eq:defprobucor201per}) comes from a collection (basically a union) of disjoint sets ${\cal X}_l,1\leq l\leq n$ where ${\cal X}_l=\{\x|\x_i\in\{0,1\},1\leq i\leq n,\|\x\|_2^2=\frac{l}{n}\}$. Now in the large $n$ limit one can then think of sets ${\cal X}_l,1\leq l\leq n$, as ${\cal X}_{\beta},0<\beta\leq 1$ (one would just need to discretize over $\beta$; that is a straightforward exercise and for the sake of keeping the exposition as free of unnecessary trivial details as possible we skip it). Now, since all quantities that we will consider below will concentrate around its mean values with overwhelming probability the union bounding over a bounded (independent of $n$) discrerized $\beta$ will affect the final results in no way. Given all of that the strategy will be to consider the feasibility of (\ref{eq:defprobucor201per}) for a fixed $\beta$ and then find the best one (in fact since we will be determining an upper bound on the storage capacity the strategy will be to find a worst $\beta$; however, this will naturally become clear as we progress with the presentation). Also, we do want to mention that it is absolutely not necessary to simplify the exposition by discretizing over $\beta$. Our entire exposition that will follow can be easily pushed through even with a variable $\beta$. However, in our view it unnecessarily complicates writings and we find the exposition way more clearer if we fix $\beta$ at the beginning and don't drag it as a variable inside all the derivations that will follow.

Now, we can go back to following further what was done in Section \ref{sec:discper} (and ultimately the strategy presented in \cite{StojnicGardGen13,StojnicGardSphNeg13,StojnicGardSphErr13}). One can then reformulate (\ref{eq:defprobucor201per}) so that the feasibility problem of interest becomes
\begin{eqnarray}
\xi_{01}=\min_{\x} \max_{\lambda\geq 0} & &  \kappa\lambda^T\1- \lambda^TH\x \nonumber \\
\mbox{subject to} & & \|\lambda\|_2= 1\nonumber \\
& & \x_i\in\left \{0,\frac{1}{\sqrt{n}}\right \},1\leq i\leq n.\label{eq:feas01per}
\end{eqnarray}
Clearly, the sign of $\xi_{01}$ determines the feasibility of (\ref{eq:defprobucor201per}). In particular, if $\xi_{01}>0$ then (\ref{eq:defprobucor201per}) is infeasible. Given the random structure of the problem (as earlier, the randomness remains over $H$) one can then pose the following probabilistic feasibility question (essentially a complete analogue to the one posed in earlier section for spherical and $\pm 1$ perceptrons): how small can $m$ be so that $\xi_{01}$ in (\ref{eq:feas01per}) is positive and (\ref{eq:defprobucor201per}) is infeasible with overwhelming probability? What follows provides an answer to such a question.

Before proceeding further we will concretize some of the strategy mentioned above. Namely, one can rewrite (\ref{eq:feas01per}) in the following way
\begin{equation}
\xi_{01}=\min_{\beta\in(0,1]} \xi_{01}(\beta),\label{eq:feas02per}
\end{equation}
where
\begin{eqnarray}
\xi_{01}(\beta)=\min_{\x} \max_{\lambda\geq 0} & &  \kappa\lambda^T\1- \lambda^TH\x \nonumber \\
\mbox{subject to} & & \|\lambda\|_2= 1\nonumber \\
& & \x_i\in\left \{0,\frac{1}{\sqrt{n}}\right \},1\leq i\leq n \nonumber \\
& & \|\x\|_2^2=\beta .\label{eq:feas01}
\end{eqnarray}
Moreover, one can scale down everything to obtain a redefined $\xi_{01}(\beta)$
\begin{eqnarray}
\xi_{01}(\beta)=\min_{\x} \max_{\lambda\geq 0} & &  \frac{\kappa}{\sqrt{\beta}}\lambda^T\1- \lambda^TH\x \nonumber \\
\mbox{subject to} & & \|\lambda\|_2= 1\nonumber \\
& & \x_i\in\left \{0,\frac{1}{\sqrt{\beta n}}\right \},1\leq i\leq n \nonumber \\
& & \|\x\|_2^2=1.\label{eq:feas01}
\end{eqnarray}
So, the strategy will be to probabilistically analyze $\xi_{01}(\beta)$ for a fixed $\beta$ and then find the $\beta$ that makes $\xi_{01}(\beta)$ the smallest possible (of course, $\xi_{01}(\beta)$ is random and one can't really be talking about it as the smallest possible; what we really mean is: given its concentrating behavior, one should find the smallest concentrating point for $\xi_{01}(\beta)$ over all $\beta$'s from $(0,1]$).

\subsection{Probabilistic analysis}
\label{sec:probanalrig01per}

In this section we will present a probabilistic analysis of the above optimization problems given in (\ref{eq:feas01}) and ultimately of the one given in (\ref{eq:feas02per}). In a nutshell, we will provide a relation between $\kappa$ and $\alpha=\frac{m}{n}$ so that with overwhelming probability over $H$ $\xi_{01}>0$. This will, of course, based on the above discussion then be enough to conclude that the problem in (\ref{eq:feas01per}) is infeasible with overwhelming probability when $\kappa$ and $\alpha=\frac{m}{n}$ satisfy such a relation.

As mentioned earlier, we will follow the analysis of the previous section. To that end we start by making use of Theorem \ref{thm:Gordonmesh1} through the following lemma (essentially an analogue to Lemma \ref{lemma:negproblemma}; the lemma is of course an easy consequence of Theorem \ref{thm:Gordonmesh1} and in fact is fairly similar to Lemma 3.1 in \cite{Gordon88}; see also \cite{StojnicHopBnds10,StojnicGardGen13} for similar considerations).
\begin{lemma}
Let $H$ be an $m\times n$ matrix with i.i.d. standard normal components. Let $\g$ and $\h$ be $m\times 1$ and $n\times 1$ vectors, respectively, with i.i.d. standard normal components. Also, let $g$ be a standard normal random variable and let $\zeta_{\lambda}$ be a function of $\x$. Then
\begin{equation}
\hspace{-.3in}P(\min_{\x_i\in\left \{0,\frac{1}{\sqrt{\beta n}}\right \},\|\x\|_2^2=1}\max_{\|\lambda\|_2=1,\lambda_i\geq 0}(-\lambda^TH\x+g-\zeta_{\lambda})\geq 0)\\\geq
P(\min_{\x_i\in\left \{0,\frac{1}{\sqrt{\beta n}}\right \},\|\x\|_2^2=1}\max_{\|\lambda\|_2=1,\lambda_i\geq 0}(\g^T\lambda+\h^T\x-\zeta_{\lambda})\geq 0).\label{eq:negproblemma01per}
\end{equation}\label{lemma:negproblemma01per}
\end{lemma}
\begin{proof}
The comment given in the proof of Lemma \ref{lemma:negproblemma} applies here as well. The difference is basically fairly minimal.
\end{proof}

Let $\zeta_{\lambda}=-\frac{\kappa}{\sqrt{\beta}}\lambda^T\1+\epsilon_{5}^{(g)}\sqrt{n}+\xi_{01}^{(l)}(\beta)$ with $\epsilon_{5}^{(g)}>0$ being an arbitrarily small constant independent of $n$. We will first look at the right-hand side of the inequality in (\ref{eq:negproblemma01per}). The following is then the probability of interest
\begin{equation}
P\left (\min_{\x_i\in\left \{0,\frac{1}{\sqrt{\beta n}}\right \},\|\x\|_2^2=1}\max_{\|\lambda\|_2=1,\lambda_i\geq 0}\left (\g^T\lambda+\h^T\x+\frac{\kappa}{\sqrt{\beta}}\lambda^T\1-\epsilon_{5}^{(g)}\sqrt{n}\right )\geq \xi_{01}^{(l)}(\beta)\right ).\label{eq:negprobanal001per}
\end{equation}
After solving the minimization over $\x$ one obtains
\begin{multline}
P\left (\min_{\x_i\in\left \{0,\frac{1}{\sqrt{\beta n}}\right \},\|\x\|_2^2=1}\max_{\|\lambda\|_2=1,\lambda_i\geq 0}\left (\g^T\lambda+\h^T\x+\frac{\kappa}{\sqrt{\beta}}\lambda^T\1-\epsilon_{5}^{(g)}\sqrt{n} \right )\geq \xi_{01}^{(l)}(\beta)\right )\\=P\left ( \|\left (\g+\frac{\kappa}{\sqrt{\beta}}\1\right )_+ \|_2-\frac{1}{\beta}\sum_{i=n-\beta n+1}^{n}\h_{(i)}-\epsilon_{5}^{(g)}\sqrt{n}\geq \xi_{01}^{(l)}(\beta)\right ),\label{eq:negprobanal101per}
\end{multline}
where $\left (\g+\frac{\kappa}{\sqrt{\beta}}\1\right )_+$ is $\left (\g+\frac{\kappa}{\sqrt{\beta}}\1\right )$ vector with negative components replaced by zeros and where $\h_{(i)}$ is vector with containing components of $\h$ sorted in non-decreasing order. Using the machinery of \cite{StojnicCSetam09} one then has
\begin{equation}
\lim_{n\rightarrow \infty}\frac{E\sum_{i=n-\beta n+1}^{n}\h_{(i)}}{n}=\frac{1}{\sqrt{2\pi}}e^{-(\mbox{erfinv}(2(1-\beta)-1))^2},\label{eq:hord01per}
\end{equation}
and
\begin{equation}
P\left (\sum_{i=n-\beta n+1}^{n}\h_{(i)}\leq (1+\epsilon_1^{(n)}) E\sum_{i=n-\beta n+1}^{n}\h_{(i)}\right )\geq 1-e^{-\epsilon_2^{(n)} n},\label{eq:conch01per}
\end{equation}
where $\epsilon_1^{(n)}$ is an arbitrarily small constant and $\epsilon_2^{(n)}$ is a constant possibly dependent on $\epsilon_1^{(n)}$ but independent of $n$.
Following line by line what was done in \cite{StojnicGardGen13} after equation $(13)$ one then has
\begin{multline}
P\left (\min_{\x_i\in\left \{0,\frac{1}{\sqrt{\beta n}}\right \},\|\x\|_2^2=1}\max_{\|\lambda\|_2=1,\lambda_i\geq 0}\left (\g^T\lambda+\h^T\x+\frac{\kappa}{\sqrt{\beta}}\lambda^T\1-\epsilon_{5}^{(g)}\sqrt{n}\right )\geq \xi_{01}^{(l)}(\beta)\right )\\\hspace{-.5in}\geq
(1-e^{-\epsilon_{2}^{(m)} m})(1-e^{-\epsilon_{2}^{(n)} n})
P\left ((1-\epsilon_{1}^{(m)})\sqrt{\alpha f_{gar}\left (\frac{\kappa}{\sqrt{\beta}}\right )}-(1+\epsilon_{1}^{(n)})\frac{1}{\sqrt{2\pi\beta}}e^{-(\mbox{erfinv}(2(1-\beta)-1))^2}-\epsilon_{5}^{(g)}\geq \frac{\xi_{01}^{(l)}(\beta)}{\sqrt{n}}\right ),
\label{eq:negprobanal201per}
\end{multline}
where as earlier
\begin{equation}
f_{gar}\left (\frac{\kappa}{\sqrt{\beta}}\right )=\frac{1}{\sqrt{2\pi}}\int_{-\frac{\kappa}{\sqrt{\beta}}}^{\infty}\left (\g_i+\frac{\kappa}{\sqrt{\beta}}\right )^2e^{-\frac{\g_i^2}{2}}d\g_i
=\frac{\kappa e^{-\frac{\kappa^2}{2\beta}}}{\sqrt{2\beta\pi}}+\frac{(\frac{\kappa^2}{\beta}+1)\mbox{erfc}\left ( -\frac{\kappa}{\sqrt{2\beta}}\right )}{2},\label{eq:fgarscaled}
\end{equation}
and $\epsilon_{5}^{(g)}$, $\epsilon_1^{(m)}$ are arbitrarily small positive constants and $\epsilon_2^{(m)}$ is a constant possibly dependent on $\epsilon_1^{(m)}$ and $f_{gar}(\frac{\kappa}{\sqrt{\beta}})$ but independent of $n$.
If
\begin{equation}
(1-\epsilon_{1}^{(m)})\sqrt{\alpha f_{gar}\left (\frac{\kappa}{\sqrt{\beta}}\right )}-(1+\epsilon_{1}^{(n)} )\frac{1}{\sqrt{2\pi\beta}}e^{-(\mbox{erfinv}(2(1-\beta)-1))^2}-\epsilon_{5}^{(g)}>\frac{\xi_{01}^{(l)}(\beta)}{\sqrt{n}},\label{eq:negcondxipu01per}
\end{equation}
one then has from (\ref{eq:negprobanal201per})
\begin{equation}
\lim_{n\rightarrow\infty}P\left (\min_{\x_i\in\left \{0,\frac{1}{\sqrt{\beta n}}\right \},\|\x\|_2^2=1}\max_{\|\lambda\|_2=1,\lambda_i\geq 0}\left (\g^T\lambda+\h^T\x+\frac{\kappa}{\sqrt{\beta}}\lambda^T\1-\epsilon_{5}^{(g)}\sqrt{n}\right )\geq \xi_{01}^{(l)}(\beta)\right )\geq 1.\label{eq:negprobanal301per}
\end{equation}

As in the previous section, we will also need the following simple estimate related to the left hand side of the inequality in (\ref{eq:negproblemma01per}). From (\ref{eq:negproblemma01per}) one has the following as the probability of interest
\begin{equation}
P\left (\min_{\x_i\in\left \{0,\frac{1}{\sqrt{\beta n}}\right \},\|\x\|_2^2=1}\max_{\|\lambda\|_2=1,\lambda_i\geq 0}\left (\frac{\kappa}{\sqrt{\beta}}\lambda^T\1-\lambda^TH\x+g-\epsilon_{5}^{(g)}\sqrt{n}-\xi_{01}^{(l)}(\beta)\right )\geq 0\right ).\label{eq:leftnegprobanal001per}
\end{equation}
Following again what was done in Section \ref{sec:probanalrigpm1} (and ultimately in \cite{StojnicGardGen13} between equations $(21)$ and $(24)$) one has,
assuming that (\ref{eq:negcondxipu01per}) holds,
\begin{multline}
\lim_{n\rightarrow\infty}P\left (\min_{\x_i\in\left \{0,\frac{1}{\sqrt{\beta n}}\right \},\|\x\|_2^2=1}\max_{\|\lambda\|_2=1,\lambda_i\geq 0}\left (\frac{\kappa}{\sqrt{\beta}}\lambda^T\1-\lambda^TH\x\right )\geq \xi_{01}^{(l)}(\beta)\right )\\\geq \lim_{n\rightarrow\infty}P\left (\min_{\x_i\in\left \{0,\frac{1}{\sqrt{\beta n}}\right \},\|\x\|_2^2=1}\max_{\|\lambda\|_2=1,\lambda_i\geq 0}\left (\g^T\y+\h^T\x+\frac{\kappa}{\sqrt{\beta}}\lambda^T\1-\epsilon_{5}^{(g)}\sqrt{n}\right )\geq \xi_{01}^{(l)}(\beta)\right )\geq 1.\label{eq:leftnegprobanal301per}
\end{multline}

We summarize the above results in the following theorem.

\begin{theorem}
Let $H$ be an $m\times n$ matrix with i.i.d. standard normal components. Let $n$ be large and let $m=\alpha n$, where $\alpha>0$ is a constant independent of $n$. Let $\xi_{01}$ be as in (\ref{eq:feas01per}) and let $\frac{\kappa}{\sqrt{\beta}}$ be a scalar constant independent of $n$. Let all $\epsilon$'s be arbitrarily small constants independent of $n$. Further, let $\g_i$ be a standard normal random variable and set
\begin{equation}
f_{gar}\left (\frac{\kappa}{\sqrt{\beta}}\right ) = \frac{1}{\sqrt{2\pi}}\int_{-\frac{\kappa}{\sqrt{\beta}}}^{\infty}\left (\g_i+\frac{\kappa}{\sqrt{\beta}}\right )^2e^{-\frac{\g_i^2}{2}}d\g_i
=\frac{\kappa e^{-\frac{\kappa^2}{2\beta}}}{\sqrt{2\beta\pi}}+\frac{(\frac{\kappa^2}{\beta}+1)\mbox{erfc}\left ( -\frac{\kappa}{\sqrt{2\beta}}\right )}{2}.\label{eq:thmerrc3001per}
\end{equation}

Let $\xi_{01}^{(l)}(\beta)$ be a scalar such that
\begin{equation}
(1-\epsilon_{1}^{(m)})\sqrt{\alpha f_{gar}\left (\frac{\kappa}{\sqrt{\beta}}\right )}-(1+\epsilon_{1}^{(n)} )\frac{1}{\sqrt{2\pi\beta}}e^{-(\mbox{erfinv}(2(1-\beta)-1))^2}-\epsilon_{5}^{(g)}>\frac{\xi_{01}^{(l)}(\beta)}{\sqrt{n}}.\label{eq:condxinthmstoc3001per}
\end{equation}
Then
\begin{equation}
\hspace{-.3in} \lim_{n\rightarrow\infty}P(\xi_{01}(\beta)\geq \xi_{01}^{(l)}(\beta))=\lim_{n\rightarrow\infty}P\left (\min_{\x_i\in\left \{0,\frac{1}{\sqrt{\beta n}}\right \},\|\x\|_2^2=1}\max_{\|\lambda\|_2=1,\lambda_i\geq 0}\left (\frac{\kappa}{\sqrt{\beta}}\lambda^T\1-\lambda^TH\x\right )\geq \xi_{01}^{(l)}(\beta)\right )\geq 1. \label{eq:probthmc3001per}
\end{equation}
Moreover, let $\xi_{01}^{(l)}$ be a scalar such that
\begin{equation}
\min_{\beta\in(0,1]}\left((1-\epsilon_{1}^{(m)})\sqrt{\alpha f_{gar}\left (\frac{\kappa}{\sqrt{\beta}}\right )}-(1+\epsilon_{1}^{(n)} )\frac{1}{\sqrt{2\pi\beta}}e^{-(\mbox{erfinv}(2(1-\beta)-1))^2}-\epsilon_{5}^{(g)}\right )>\min_{\beta\in(0,1]}\frac{\xi_{01}^{(l)}(\beta)}{\sqrt{n}}
=\frac{\xi_{01}^{(l)}}{\sqrt{n}}.\label{eq:condxinthmstoc3001pernobeta}
\end{equation}
Then
\begin{equation}
\lim_{n\rightarrow\infty}P(\xi_{01}\geq \xi_{01}^{(l)})=\lim_{n\rightarrow\infty}P\left (\min_{\x_i\in\left \{0,\frac{1}{\sqrt{ n}}\right \},\|\x\|_2^2=1}\max_{\|\lambda\|_2=1,\lambda_i\geq 0}\left (\frac{\kappa}{\sqrt{\beta}}\lambda^T\1-\lambda^TH\x\right )\geq \xi_{01}^{(l)}\right )\geq 1 \label{eq:probthmc3001pernobeta}
\end{equation}
and (\ref{eq:defprobucor201per}) is infeasible with overwhelming probability.
\label{thm:01perc30}
\end{theorem}
\begin{proof}
Follows from the above discussion, comments right after (\ref{eq:defprobucor201per}), and the analysis presented in \cite{StojnicGardGen13}.
\end{proof}

In a more informal language (as earlier, essentially ignoring all technicalities and $\epsilon$'s) one has that as long as
\begin{equation}
\alpha>\max_{\beta\in(0,1]}\left (\frac{e^{-2(\mbox{erfinv}(2(1-\beta)-1))^2}}{2\pi\beta f_{gar}\left (\frac{\kappa}{\sqrt{\beta}}\right )}\right ),\label{eq:condalph01per}
\end{equation}
the problem in (\ref{eq:defprobucor201per}) will be infeasible with overwhelming probability. It is an easy exercise to show that the above is exactly the prediction for the storage capacity given in \cite{GutSte90} for $0/1$ perceptron. This basically establishes the prediction obtained based on the replica symmetry approach of statistical mechanics as a rigorous upper bound the true value of the storage capacity of $0/1$ preceptron.

The results obtained based on the above theorem (as well as those predicted assuming replica symmetry and given in \cite{GutSte90}) are presented in Figure \ref{fig:01perc30}. For the values of $\alpha$ that are to the right of the given curve the memory will not operate correctly with overwhelming probability. This of course follows from the fact that with overwhelming probability over $H$ the inequalities in (\ref{eq:defprobucor201per}) will not be simultaneously satisfiable.
\begin{figure}[htb]
\centering
\centerline{\epsfig{figure=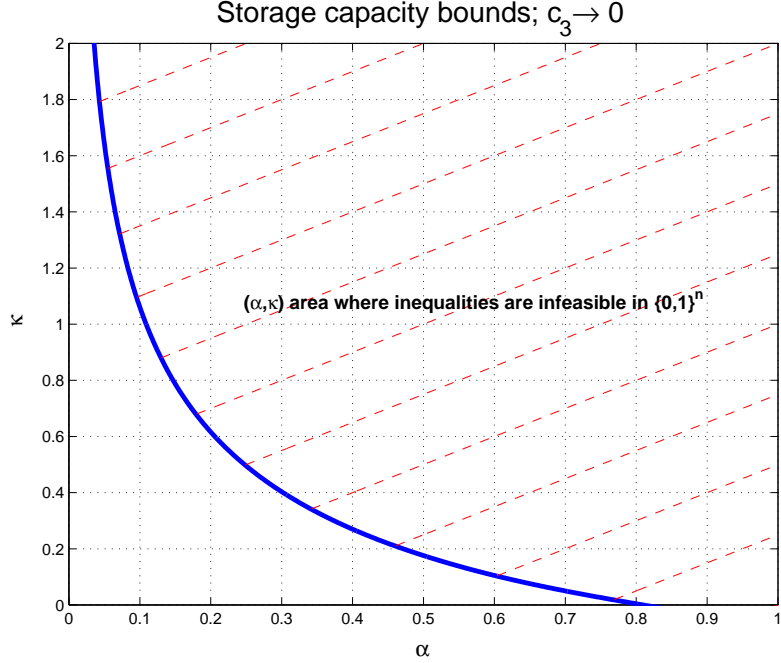,width=10.5cm,height=9cm}}
\caption{$\frac{\kappa}{\sqrt{\beta}}$ as a function of $\alpha$; $\x\in\left \{0,\frac{1}{\sqrt{n}}\right \}^n$}
\label{fig:01perc30}
\end{figure}

We should also mention that one can employ the technique similar to the one presented in Section \ref{sec:discperlow} to attempt to lower the upper bounds presented in Figure \ref{fig:01perc30}. However, since we haven't found a substantial improvement over the results already presented in Figure \ref{fig:01perc30} we skip presenting results in that direction and instead present a simple combinatorial upper bound that can be obtained following the approach presented in Section \ref{sec:rigresdiscsimpcomb}.

\subsection{Simple combinatorial bound -- $0/1$ perceptron}
\label{sec:rigres01simpcomb}

In this section we will briefly sketch how one can obtain results for $0/1$ perceptron that are similar to those presented in Section \ref{sec:rigresdiscsimpcomb} for $\pm 1$ perceptron.

As in Section \ref{sec:rigresdiscsimpcomb} one starts by looking at how likely is that each of the inequalities in (\ref{eq:defprobucor201per}) is satisfied. Similarly to what we did in the previous subsection we will fix a $\beta\in [0,1]$ and consider $\x_i\in\left \{0,\frac{1}{\sqrt{\beta n}}\right \}$ such that $\|\x\|_2=1$. Following what was done in Section \ref{sec:rigresdiscsimpcomb} one then has
\begin{equation}
P\left (H_{i,:}\x\geq \frac{\kappa}{\sqrt{\beta}}|\|\x\|_2=1,\x_i\in\left \{0,\frac{1}{\sqrt{\beta n}}\right \}\right )=P(g\geq \frac{\kappa}{\sqrt{\beta}})=\frac{1}{2}\mbox{erfc}(\frac{\kappa}{\sqrt{2\beta}}),1\leq i\leq m.\label{eq:simpcomb101per}
\end{equation}
After accounting for all the inequalities in (\ref{eq:defprobucor201per}) (essentially all the rows of $H$) one then further has
\begin{equation}
P\left (H\x\geq \frac{\kappa}{\sqrt{\beta}}|\|\x\|_2=1,\x_i\in\left \{0,\frac{1}{\sqrt{\beta n}}\right \}\right )=\left (P\left (H_{i,:}\x\geq \frac{\kappa}{\sqrt{\beta}}|\|\x\|_2=1,\x_i\in\left \{0,\frac{1}{\sqrt{\beta n}}\right \}\right )\right )^m.\label{eq:simpcomb201per}
\end{equation}
Using the union bound over all $\x$ then gives
\begin{eqnarray}
\hspace{-.4in}P\left (\exists \x|H\x\geq \frac{\kappa}{\sqrt{\beta}},\|\x\|_2=1,\x_i\in\left \{0,\frac{1}{\sqrt{\beta n}}\right \}\right )& \leq &
e^{h(\beta)n}P\left (H\x\geq \frac{\kappa}{\sqrt{\beta}}|\|\x\|_2=1,\x_i\in\left \{0,\frac{1}{\sqrt{\beta n}}\right \}\right ) \nonumber \\& = &
e^{h(\beta)n}\left (P\left (H_{i,:}\x\geq \frac{\kappa}{\sqrt{\beta}}|\|\x\|_2=1,\x_i\in\left \{0,\frac{1}{\sqrt{\beta n}}\right \}\right )\right )^m,\nonumber \\\label{eq:simpcomb301per}
\end{eqnarray}
where $h()$ is the entropy function of basis $e$, i.e.
\begin{equation}
h(\beta)=\beta\log(\beta)+(1-\beta)\log(1-\beta).\label{eq:defent}
\end{equation}
A combination of (\ref{eq:simpcomb101per}) and (\ref{eq:simpcomb301per}) then gives
\begin{equation}
P\left (\exists \x|H\x\geq \frac{\kappa}{\sqrt{\beta}},\|\x\|_2=1,\x_i\in\left \{0,\frac{1}{\sqrt{\beta n}}\right \}\right )\leq e^{-h(\beta)n}\left (\frac{1}{2}\mbox{erfc}(\frac{\kappa}{\sqrt{2\beta}})\right )^m.\label{eq:simpcomb401per}
\end{equation}
After discretizing over $\beta$ and union bounding one from (\ref{eq:simpcomb401per}) then has that if $\alpha=\frac{m}{n}$ is such that
\begin{equation}
\alpha> \max_{\beta\in(0,1]}\frac{h(\beta)}{\log\left (\frac{1}{2}\mbox{erfc}(\frac{\kappa}{\sqrt{2\beta}})\right )},\label{eq:simpcomb501per}
\end{equation}
then
\begin{equation}
\lim_{n\rightarrow\infty}P\left (\exists \x|H\x\geq \kappa,\x_i\in\left \{0,\frac{1}{\sqrt{n}}\right \}\right )\leq \lim_{n\rightarrow\infty}
\max_{\beta\in(0,1]}\left (e^{-h(\beta)n}\left (\frac{1}{2}\mbox{erfc}(\frac{\kappa}{\sqrt{2\beta}})\right )^m\right )=0.\label{eq:simpcomb601per}
\end{equation}
The upper bounds on the storage capacity one can obtain based on the above consideration (in particular based on (\ref{eq:simpcomb501per})) are presented in Figure \ref{fig:01persimpcomb}. Similarly to what we mentioned when we studied $\pm 1$ perceptrons, while these bounds can be improved, our goal is more to recall on the results that relate to the ones that we present in this paper rather than on the best possible ones.
\begin{figure}[htb]
\centering
\centerline{\epsfig{figure=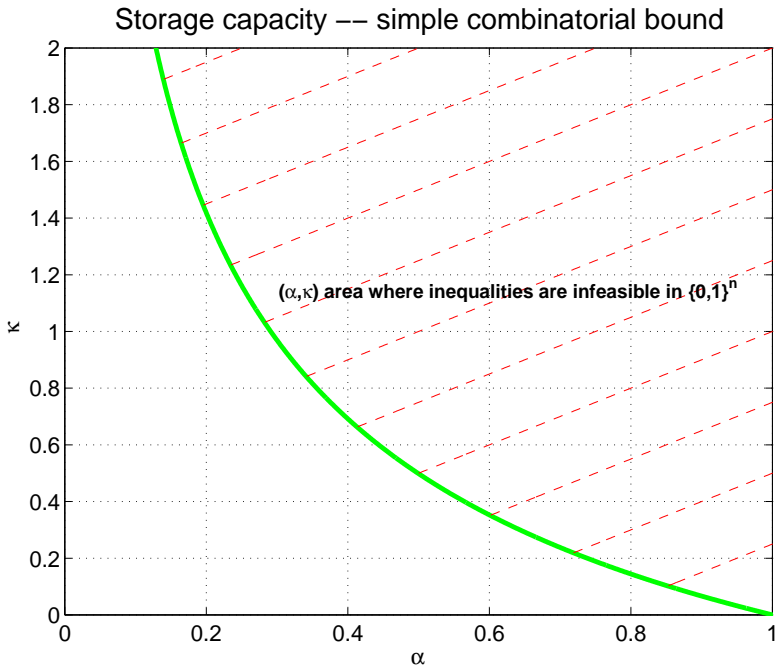,width=10.5cm,height=9cm}}
\caption{$\kappa$ as a function of $\alpha$; simple combinatorial bound; $\x\in\left \{0,\frac{1}{\sqrt{n}}\right \}^n$}
\label{fig:01persimpcomb}
\end{figure}

Also, in Figure \ref{fig:01peroptc3} we present the above simple combinatorial bounds together with the results obtained in the previous subsection. Differently from what was the case when we studied $\pm 1$ perceptron, here the simple combinatorial bound does not seem to improve over the results we presented in Section \ref{sec:probanalrig01per} (at least not in the range of $\kappa$'s that we considered). We also indicate in Figure \ref{fig:01peroptc3} that even if one is to employ the strategy from Section \ref{sec:discperlow} the optimal corresponding $c_3^{(s)}$ would turn out to be the one converging to zero.
\begin{figure}[htb]
\centering
\centerline{\epsfig{figure=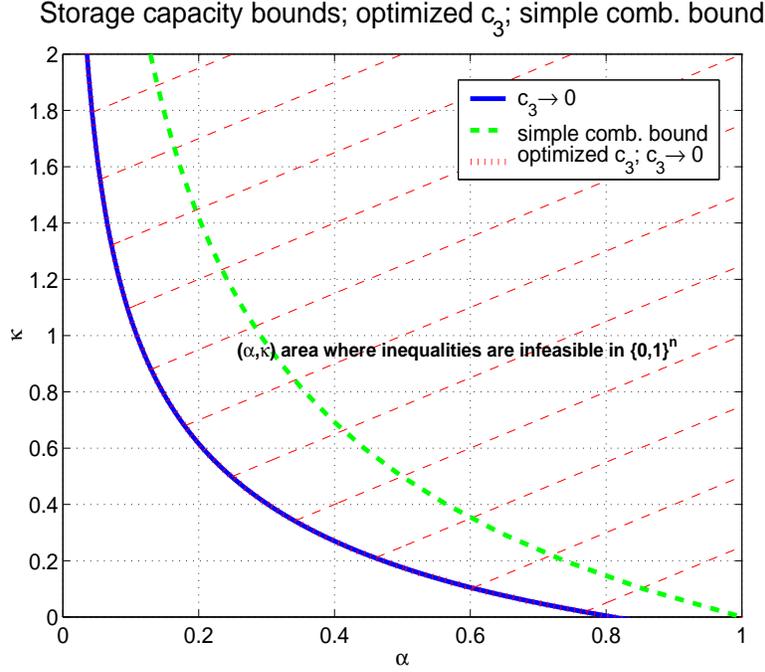,width=10.5cm,height=9cm}}
\caption{$\kappa$ as a function of $\alpha$; optimized $c_3^{(s)}$; $\x\in\left \{0,\frac{1}{\sqrt{n}}\right \}^n$}
\label{fig:01peroptc3}
\end{figure}
We should also mention that (as was the case for the $\pm 1$ perceptron) the predictions based on the zero entropy obtained in \cite{GutSte90} are still substantially lower than the ones presented in Figure \ref{fig:01peroptc3}. For example, for $\kappa=0$ the above results guarantee that $\alpha\leq 0.809$ (which is the same what the replica symmetry theory predicts) whereas the zero entropy calculations of \cite{GutSte90} predict $\alpha\approx 0.59$.

As we have mentioned in the introduction various types of discrete perceptrons are possible. Above we chose the two fairly typical ones: the $\pm 1$ and the $0/1$ perceptron. Many others have been discussed/analyzed throughout the vast perceptron literature , see, e.g. \cite{GutSte90}. Among them are more general versions of $\pm 1$ such as the one where $\x_i\in\left \{\pm\frac{L}{\sqrt{n}},\pm\frac{L-1}{\sqrt{n}},\dots,\pm\frac{1}{\sqrt{n}} \right \}$ or its a slight alternation
where $\x_i\in\left \{\pm\frac{L}{\sqrt{n}},\pm\frac{L-1}{\sqrt{n}},\dots,\pm\frac{1}{\sqrt{n}},0 \right \}$. These are referred to as the digital perceptrons in \cite{GutSte90}. The strategies designed above can easily be adapted to handle these cases as well. However, as we have mentioned earlier, to preserve the elegance of the exposition, we chose only two particular cases to demonstrate how the concepts work and left the remaining scenarios for a more technical presentation. However, we also chose one extra case that goes on top of those mentioned above. Such a case is essentially a limiting case of digital perceptrons obtained in the limit of large $L$. Basically, as $L$ grows the digital perceptrons should converge to the so-called box-constrained perceptrons where $\x_i\in [-1,1]$. An interesting phenomenon happens in the analysis of such perceptrons and that is of course the reason why we selected it. We will present the results related to the box-constrained perceptrons in the following section.

\section{Box-constrained perceptrons}
\label{sec:boxper}

As mentioned above, in this section we look at the box-constrained perceptrons. To make the presentation as easy to follow as possible we will again try to emulate the exposition of Sections \ref{sec:discper} and \ref{sec:01per} as much as possible. Following what was done in Sections \ref{sec:discper} and \ref{sec:01per} it is not that hard to recognize that the storage capacity of box-constrained perceptron can be considered through the following feasibility problem
\begin{eqnarray}
& & H\x\geq \kappa\nonumber \\
& & \x_i\in \left [-\frac{1}{\sqrt{n}},\frac{1}{\sqrt{n}}\right ],1\leq i\leq n.\label{eq:defprobucor2boxper}
\end{eqnarray}
As in Sections \ref{sec:discper} and \ref{sec:01per} to ease the exposition we will continue assume that the elements of $H$ are i.i.d. standard normals and that the dimension of $H$ is $m\times n$. Moreover, we will continue to work in the linear regime, i.e. we will continue to assume that $m=\alpha n$ where $\alpha$ is a constant independent of $n$. Now, if all inequalities in (\ref{eq:defprobucor2boxper}) are satisfied one can have that the perceptron dynamics discussed in Section \ref{sec:mathsetupper} will be stable and all $m$ patterns could be successfully stored.

As was the case in Section \ref{sec:01per}, before proceeding further with following the exposition of the previous section, we will first discretize the problem a bit. It is not that hard to see that the set of all allowed $\x$ in (\ref{eq:defprobucor2boxper}) comes from a collection (basically a union) of disjoint sets ${\cal X}_{\beta}^{(box)},0< \beta\leq 1$, where
\begin{equation}
{\cal X}_l^{(box)}=\left \{\x|\x_i\in\left [-\frac{1}{\sqrt{n}},\frac{1}{\sqrt{n}}\right ],1\leq i\leq n,\|\x\|_2^2=\beta,\beta\in(0,1]\right \}.\label{eq:defcalXbox}
\end{equation}
As we discussed in the previous section, one would need to discretize over $\beta$ (as mentioned in the previous section, that is a fairly straightforward and we skip it). Since all quantities that we will consider below will concentrate around its mean values with overwhelming probability the union bounding over a bounded (independent of $n$) discrerized $\beta$ will affect the final results in no way. Given all of that the strategy will be similar to the one from the previous section. Basically, we will consider the feasibility of (\ref{eq:defprobucor2boxper}) for a fixed $\beta$ and then find optimize to find the best/worst one. this is again absolutely not necessary. As in the previous section, our entire exposition that will follow can be easily pushed through even with a variable $\beta$. However, in our view it unnecessarily complicates writings and we find the exposition way more clearer if we again fix $\beta$ at the beginning and don't drag it as a variable inside all the derivations that will follow.

Going back to (\ref{eq:defprobucor2boxper}) one can rewrite it as the following problem
\begin{eqnarray}
\xi_{box}=\min_{\x} \max_{\lambda\geq 0} & &  \kappa\lambda^T\1- \lambda^TH\x \nonumber \\
\mbox{subject to} & & \|\lambda\|_2= 1\nonumber \\
& & \x_i\in\left [-\frac{1}{\sqrt{n}},\frac{1}{\sqrt{n}}\right ],1\leq i\leq n.\label{eq:feasboxper}
\end{eqnarray}
As earlier, the critical component of the analysis that will follow will be the sign of $\xi_{box}$. Obviously, the sign of $\xi_{box}$ determines the feasibility of (\ref{eq:defprobucor2boxper}). In particular, if $\xi_{box}>0$ then (\ref{eq:defprobucor2boxper}) is infeasible and if $\xi_{box}\geq 0$ then (\ref{eq:defprobucor2boxper}) is feasible. One can then ask the following analogue to the probabilistic questions asked in Sections \ref{sec:discper} and \ref{sec:01per}: how small can $m$ be so that $\xi_{box}$ in (\ref{eq:feasboxper}) is positive and (\ref{eq:defprobucor2boxper}) is infeasible with overwhelming probability? And, how large can $m$ be so that $\xi_{box}$ in (\ref{eq:feasboxper}) is negative and (\ref{eq:defprobucor2boxper}) is feasible with overwhelming probability? (As usual, we recall that the overwhelming probability is over the randomness of $H$). Below  we provide the exact answers to these questions.

Before proceeding further we will need a few technical details setup. They relate to the concretizing the above mentioned dealing with $\beta$. We will do so by rewriting (\ref{eq:feasboxper}) in the following way
\begin{equation}
\xi_{box}=\min_{\beta\in(0,1]} \xi_{box}(\beta),\label{eq:feasboxper1}
\end{equation}
where
\begin{eqnarray}
\xi_{box}(\beta)=\min_{\x} \max_{\lambda\geq 0} & &  \kappa\lambda^T\1- \lambda^TH\x \nonumber \\
\mbox{subject to} & & \|\lambda\|_2= 1\nonumber \\
& & \x_i\in\left [-\frac{1}{\sqrt{n}},\frac{1}{\sqrt{n}}\right ],1\leq i\leq n \nonumber \\
& & \|\x\|_2^2=\beta .\label{eq:feasbox}
\end{eqnarray}
Following further what was done in the previous section, one can scale down everything to obtain a redefined $\xi_{box}(\beta)$
\begin{eqnarray}
\xi_{box}(\beta)=\min_{\x} \max_{\lambda\geq 0} & &  \frac{\kappa}{\sqrt{\beta}}\lambda^T\1- \lambda^TH\x \nonumber \\
\mbox{subject to} & & \|\lambda\|_2= 1\nonumber \\
& & \x_i\in\left [-\frac{1}{\sqrt{\beta n}},\frac{1}{\sqrt{\beta n}}\right ],1\leq i\leq n \nonumber \\
& & \|\x\|_2^2=1.\label{eq:feasbox1}
\end{eqnarray}
The above mentioned strategy will be then boil down to a probabilistic analysis of $\xi_{box}(\beta)$ for a fixed $\beta$. Then we will try to find the $\beta$ that makes $\xi_{box}(\beta)$ the smallest possible (as in the previous section, what we mean is: given its concentrating behavior, we will try to find the smallest concentrating points for $\xi_{box}(\beta)$ over all $\beta$'s from $(0,1]$).

\subsection{Probabilistic analysis}
\label{sec:probanalrigboxper}

In this section we will present a probabilistic analysis of the above optimization problems given in (\ref{eq:feasbox}) and ultimately of the one given in (\ref{eq:feasboxper}). In a nutshell, in the first part below we will provide a relation between $\kappa$ and $\alpha=\frac{m}{n}$ so that with overwhelming probability over $H$ $\xi_{box}>0$. This will, of course, based on the above discussion then be enough to conclude that the problem in (\ref{eq:feasboxper}) is infeasible with overwhelming probability when $\kappa$ and $\alpha=\frac{m}{n}$ satisfy such a relation. In the second part we will then provide a relation between $\kappa$ and $\alpha=\frac{m}{n}$ so that with overwhelming probability over $H$ $\xi_{box}\geq0$. This will then be enough to conclude that the problem in (\ref{eq:feasboxper}) is feasible with overwhelming probability when $\kappa$ and $\alpha=\frac{m}{n}$ satisfy such a relation. Moreover, the two relation between $\kappa$ and $\alpha=\frac{m}{n}$ will pretty much match each other.

\subsubsection{Lower-bounding $\xi_{box}$}
\label{sec:xiboxlb}

We will again to a degree follow the analysis of the previous sections. We start with the following analogue to Lemmas \ref{lemma:negproblemma} and \ref{lemma:negproblemma01per} (the lemma is of course an easy consequence of Theorem \ref{thm:Gordonmesh1} and in fact is fairly similar to Lemma 3.1 in \cite{Gordon88}).
\begin{lemma}
Let $H$ be an $m\times n$ matrix with i.i.d. standard normal components. Let $\g$ and $\h$ be $m\times 1$ and $n\times 1$ vectors, respectively, with i.i.d. standard normal components. Also, let $g$ be a standard normal random variable and let $\zeta_{\lambda}$ be a function of $\x$. Then
\begin{multline}
P\min_{\x_i\in\left [-\frac{1}{\sqrt{\beta n}},\frac{1}{\sqrt{\beta n}}\right ],\|\x\|_2^2=1}\max_{\|\lambda\|_2=1,\lambda_i\geq 0} (-\lambda^TH\x+g-\zeta_{\lambda} )\geq 0 )\\\geq
P (\min_{\x_i\in\left [-\frac{1}{\sqrt{\beta n}},\frac{1}{\sqrt{\beta n}}\right ],\|\x\|_2^2=1}\max_{\|\lambda\|_2=1,\lambda_i\geq 0} (\g^T\lambda+\h^T\x-\zeta_{\lambda})\geq 0 ).\label{eq:negproblemmaboxper}
\end{multline}\label{lemma:negproblemmaboxper}
\end{lemma}
\begin{proof}
The comment given in the proof of Lemma \ref{lemma:negproblemma} applies here as well. The difference is again basically fairly minimal.
\end{proof}

Let $\zeta_{\lambda}=-\frac{\kappa}{\sqrt{\beta}}\lambda^T\1+\epsilon_{5}^{(g)}\sqrt{n}+\xi_{box}^{(l)}(\beta)$ with $\epsilon_{5}^{(g)}>0$ being an arbitrarily small constant independent of $n$. We will first look at the right-hand side of the inequality in (\ref{eq:negproblemmaboxper}). The following is then the probability of interest
\begin{equation}
P\left (\min_{\x_i\in\left [-\frac{1}{\sqrt{\beta n}},\frac{1}{\sqrt{\beta n}}\right ],\|\x\|_2^2=1}\max_{\|\lambda\|_2=1,\lambda_i\geq 0}\left (\g^T\lambda+\h^T\x+\frac{\kappa}{\sqrt{\beta}}\lambda^T\1-\epsilon_{5}^{(g)}\sqrt{n}\right )\geq \xi_{box}^{(l)}(\beta)\right ).\label{eq:negprobanal001per}
\end{equation}
After solving the maximization over $\lambda$ one obtains
\begin{multline}
P\left (\min_{\x_i\in\left [-\frac{1}{\sqrt{\beta n}},\frac{1}{\sqrt{\beta n}}\right ],\|\x\|_2^2=1}\max_{\|\lambda\|_2=1,\lambda_i\geq 0}\left (\g^T\lambda+\h^T\x+\frac{\kappa}{\sqrt{\beta}}\lambda^T\1-\epsilon_{5}^{(g)}\sqrt{n} \right )\geq \xi_{box}^{(l)}(\beta)\right )\\=P\left ( \|\left (\g+\frac{\kappa}{\sqrt{\beta}}\1\right )_+ \|_2+\min_{\x_i\in\left [-\frac{1}{\sqrt{\beta n}},\frac{1}{\sqrt{\beta n}}\right ],\|\x\|_2^2=1}\h^T\x-\epsilon_{5}^{(g)}\sqrt{n}\geq \xi_{box}^{(l)}(\beta)\right ),\label{eq:negprobanal101per}
\end{multline}
where $\left (\g+\frac{\kappa}{\sqrt{\beta}}\1\right )_+$ is $\left (\g+\frac{\kappa}{\sqrt{\beta}}\1\right )$ vector with negative components replaced by zeros and where
\begin{equation}
f_{box}^{(r)}(\h,\beta)=\min_{\x_i\in\left [-\frac{1}{\sqrt{\beta n}},\frac{1}{\sqrt{\beta n}}\right ],\|\x\|_2^2=1}\h^T\x.\label{eq:defrfbox}
\end{equation}
Due to the linearity of the objective function in the definition of $f_{box}^{(r)}(\h)$ and the fact that $\h$ is a vector of $n$ i.i.d. standard normals, one has
\begin{equation}
P(f_{box}^{(r)}(\h,\beta)>(1+\epsilon_{1}^{(n)})f_{box}(\beta)\sqrt{n})\geq 1-e^{-\epsilon_{2}^{(n)} n},\label{eq:devhboxper}
\end{equation}
where
\begin{equation}
f_{box}(\beta)=\lim_{n\rightarrow \infty}\frac{Ef_{box}^{(r)}(\h,\beta)}{\sqrt{n}}=\lim_{n\rightarrow\infty}\frac{E\left ( \min_{\x_i\in\left [-\frac{1}{\sqrt{\beta n}},\frac{1}{\sqrt{\beta n}}\right ],\|\x\|_2^2=1}\h^T\x\right )}{\sqrt{n}},\label{eq:fbox}
\end{equation}
and $\epsilon_{1}^{(n)}>0$ is an arbitrarily small constant and analogously as above $\epsilon_{2}^{(n)}$ is a constant dependent on $\epsilon_{1}^{(n)}$ and
$f_{box}(\h,\beta)$ but independent of $n$.
Following what was done in the previous sections one then has
\begin{multline}
P\left (\min_{\x_i\in\left [-\frac{1}{\sqrt{\beta n}},\frac{1}{\sqrt{\beta n}}\right ],\|\x\|_2^2=1}\max_{\|\lambda\|_2=1,\lambda_i\geq 0}\left (\g^T\lambda+\h^T\x+\frac{\kappa}{\sqrt{\beta}}\lambda^T\1-\epsilon_{5}^{(g)}\sqrt{n}\right )\geq \xi_{box}^{(l)}(\beta)\right )\\\hspace{-.5in}\geq
(1-e^{-\epsilon_{2}^{(m)} m})(1-e^{-\epsilon_{2}^{(n)} n})
P\left ((1-\epsilon_{1}^{(m)})\sqrt{\alpha f_{gar}\left (\frac{\kappa}{\sqrt{\beta}}\right )}+(1+\epsilon_{1}^{(n)})f_{box}(\beta)-\epsilon_{5}^{(g)}\geq \frac{\xi_{box}^{(l)}(\beta)}{\sqrt{n}}\right ),
\label{eq:negprobanal2boxper}
\end{multline}
where as earlier
\begin{equation}
f_{gar}\left (\frac{\kappa}{\sqrt{\beta}}\right )=\frac{1}{\sqrt{2\pi}}\int_{-\frac{\kappa}{\sqrt{\beta}}}^{\infty}\left (\g_i+\frac{\kappa}{\sqrt{\beta}}\right )^2e^{-\frac{\g_i^2}{2}}d\g_i
=\frac{\kappa e^{-\frac{\kappa^2}{2\beta}}}{\sqrt{2\beta\pi}}+\frac{(\frac{\kappa^2}{\beta}+1)\mbox{erfc}\left ( -\frac{\kappa}{\sqrt{2\beta}}\right )}{2},\label{eq:fgarscaledboxper}
\end{equation}
and $\epsilon_{5}^{(g)}$, $\epsilon_1^{(m)}$ are arbitrarily small positive constants and $\epsilon_2^{(m)}$ is a constant possibly dependent on $\epsilon_1^{(m)}$ and $f_{gar}(\frac{\kappa}{\sqrt{\beta}})$ but independent of $n$.
If
\begin{equation}
(1-\epsilon_{1}^{(m)})\sqrt{\alpha f_{gar}\left (\frac{\kappa}{\sqrt{\beta}}\right )}+(1+\epsilon_{1}^{(n)} )f_{box}(\beta)-\epsilon_{5}^{(g)}>\frac{\xi_{box}^{(l)}(\beta)}{\sqrt{n}},\label{eq:negcondxipuboxper}
\end{equation}
one then has from (\ref{eq:negprobanal2boxper})
\begin{equation}
\lim_{n\rightarrow\infty}P\left (\min_{\x_i\in\left [-\frac{1}{\sqrt{\beta n}},\frac{1}{\sqrt{\beta n}}\right ],\|\x\|_2^2=1}\max_{\|\lambda\|_2=1,\lambda_i\geq 0}\left (\g^T\lambda+\h^T\x+\frac{\kappa}{\sqrt{\beta}}\lambda^T\1-\epsilon_{5}^{(g)}\sqrt{n}\right )\geq \xi_{box}^{(l)}(\beta)\right )\geq 1.\label{eq:negprobanal3boxper}
\end{equation}

As in previous sections, we will also need the following simple estimate related to the left hand side of the inequality in (\ref{eq:negproblemmaboxper}). From (\ref{eq:negproblemmaboxper}) one has the following as the probability of interest
\begin{equation}
P\left (\min_{\x_i\in\left [-\frac{1}{\sqrt{\beta n}},\frac{1}{\sqrt{\beta n}}\right ],\|\x\|_2^2=1}\max_{\|\lambda\|_2=1,\lambda_i\geq 0}\left (\frac{\kappa}{\sqrt{\beta}}\lambda^T\1-\lambda^TH\x+g-\epsilon_{5}^{(g)}\sqrt{n}-\xi_{box}^{(l)}(\beta)\right )\geq 0\right ).\label{eq:leftnegprobanal0boxper}
\end{equation}
Following again what was done in Section \ref{sec:probanalrigpm1} (and ultimately in \cite{StojnicGardGen13} between equations $(21)$ and $(24)$) one has,
assuming that (\ref{eq:negcondxipuboxper}) holds,
\begin{multline}
\lim_{n\rightarrow\infty}P\left (\min_{\x_i\in\left [-\frac{1}{\sqrt{\beta n}},\frac{1}{\sqrt{\beta n}}\right ],\|\x\|_2^2=1}\max_{\|\lambda\|_2=1,\lambda_i\geq 0}\left (\frac{\kappa}{\sqrt{\beta}}\lambda^T\1-\lambda^TH\x\right )\geq \xi_{box}^{(l)}(\beta)\right )\\\geq \lim_{n\rightarrow\infty}P\left (\min_{\x_i\in\left [-\frac{1}{\sqrt{\beta n}},\frac{1}{\sqrt{\beta n}}\right ],\|\x\|_2^2=1}\max_{\|\lambda\|_2=1,\lambda_i\geq 0}\left (\g^T\y+\h^T\x+\frac{\kappa}{\sqrt{\beta}}\lambda^T\1-\epsilon_{5}^{(g)}\sqrt{n}\right )\geq \xi_{box}^{(l)}(\beta)\right )\geq 1.\label{eq:leftnegprobanal3boxper}
\end{multline}
To have the above strategy operational one needs an estimate on $f_{box}(\beta)$. In the following subsection we present a way to obtain such an estimate.

\subsubsection{Estimating $f_{box}(\beta)$}
\label{sec:estfbox}

In this subsection we look at $f_{box}(\beta)$. It is relatively easy to see that a lower bound on $f_{box}(\beta)$ will enable the above machinery to work (we will actually determine more than that but for the purposes we need here a lower bound would be sufficient). Instead of directly looking at $f_{box}(\beta)$ we start actually by first looking at $f_{box}^{(r)}(\h,\beta)$. To that end we recall that from (\ref{eq:defrfbox})
\begin{equation}
f_{box}^{(r)}(\h,\beta)=\min_{\x_i\in\left [-\frac{1}{\sqrt{\beta n}},\frac{1}{\sqrt{\beta n}}\right ],\|\x\|_2^2=1}\h^T\x.\label{eq:defrfbox1}
\end{equation}
One then easily has
\begin{equation}
f_{box}^{(r)}(\h,\beta)=\frac{1}{\sqrt{\beta n}}\min_{\x_i\in\left [-1,1\right ],\|\x\|_2^2=\beta n}\h^T\x.\label{eq:defrfbox2}
\end{equation}
The following line of identities/inequalities is also easy to establish
\begin{eqnarray}
f_{box}^{(r)}(\h,\beta) & = & \frac{1}{\sqrt{\beta n}}\min_{\x_i\in\left [-1,1\right ],\|\x\|_2^2=\beta n}\h^T\x\nonumber \\
& = & \frac{1}{\sqrt{\beta n}}\min_{\x_i\in\left [-1,1\right ]}\max_{\gamma\geq 0}(\h^T\x+\gamma\|\x\|_2^2-\gamma\beta n)\nonumber \\
& \geq & \frac{1}{\sqrt{\beta n}} \max_{\gamma\geq 0}\min_{\x_i\in\left [-1,1\right ]}(\h^T\x+\gamma\|\x\|_2^2-\gamma\beta n)\nonumber \\
& = & \frac{1}{\sqrt{\beta n}} \max_{\gamma\geq 0}(\sum_{i=1}^{n}f_{box}^{(r,1)}(\h_i,\gamma)-\gamma\beta n),\label{eq:defrfbox3}
\end{eqnarray}
where
\begin{equation}
f_{box}^{(r,1)}(\h_i,\gamma)=\begin{cases}\h_i+\gamma, & \h_i\leq -2\gamma\\
-\frac{\h_i^2}{4\gamma}, & |\h_i|\leq 2\gamma\\
-\h_i+\gamma, & \h_i\geq 2\gamma.
\end{cases}\label{eq:casesfbox}
\end{equation}
Although we don't need it here, we do mention that the strong duality holds and the inequality can be replaced with an equality. Combining (\ref{eq:fbox}), (\ref{eq:defrfbox2}), and (\ref{eq:defrfbox3}) one then has
\begin{eqnarray}
f_{box}(\beta)=\lim_{n\rightarrow \infty}\frac{Ef_{box}^{(r)}(\h,\beta)}{\sqrt{n}} & = &\lim_{n\rightarrow\infty}\frac{E\left ( \min_{\x_i\in\left [-\frac{1}{\sqrt{\beta n}},\frac{1}{\sqrt{\beta n}}\right ],\|\x\|_2^2=1}\h^T\x\right )}{\sqrt{n}}\nonumber \\
& = & \frac{1}{\sqrt{\beta}}E\max_{\gamma\geq 0}(f_{box}^{(r,1)}(\h_i,\gamma)-\gamma\beta )\nonumber \\
& \geq  & \frac{1}{\sqrt{\beta}}\max_{\gamma\geq 0}(Ef_{box}^{(r,1)}(\h_i,\gamma)-\gamma\beta ).\label{eq:fbox1}
\end{eqnarray}
The last inequality can be replaced by an inequality. For what we need here though the inequality suffices (however, one should keep this as well as the above mentioned strong duality point in mind since they will be of use in the next subsection). After solving the integrals one finds
\begin{equation}
Ef_{box}^{(r,1)}(\h_i,\gamma)=I_1^{(box)}+I_2^{(box)},\label{eq:Efbox}
\end{equation}
where
\begin{eqnarray}
I_1^{(box)} & = & -\frac{2e^{-2\gamma^2}}{\sqrt{2\pi}}+\gamma\mbox{erfc}\left (\frac{2\gamma}{\sqrt{2}}\right )\nonumber \\
I_2^{(box)} & = & -\frac{1}{2\gamma}\left (-\frac{2\gamma e^{-2\gamma^2}}{\sqrt{2\pi}}+\frac{1}{2}\left (\mbox{erfc}\left (-\frac{2\gamma}{\sqrt{2}}\right )-1\right )\right).
\end{eqnarray}

We summarize the results from this and previous subsection in the following theorem.

\begin{theorem}
Let $H$ be an $m\times n$ matrix with i.i.d. standard normal components. Let $n$ be large and let $m=\alpha n$, where $\alpha>0$ is a constant independent of $n$. Let $\xi_{box}$ be as in (\ref{eq:feasboxper}) and let $\kappa>0$ be a scalar constant independent of $n$. Let all $\epsilon$'s be arbitrarily small constants independent of $n$. Further, let $\g_i$ be a standard normal random variable and set
\begin{equation}
f_{gar}\left (\frac{\kappa}{\sqrt{\beta}}\right ) = \frac{1}{\sqrt{2\pi}}\int_{-\frac{\kappa}{\sqrt{\beta}}}^{\infty}\left (\g_i+\frac{\kappa}{\sqrt{\beta}}\right )^2e^{-\frac{\g_i^2}{2}}d\g_i
=\frac{\kappa e^{-\frac{\kappa^2}{2\beta}}}{\sqrt{2\beta\pi}}+\frac{(\frac{\kappa^2}{\beta}+1)\mbox{erfc}\left ( -\frac{\kappa}{\sqrt{2\beta}}\right )}{2},\label{eq:thmfgarboxper}
\end{equation}
and
\begin{equation}
\widehat{f_{box}}(\beta)=\frac{1}{\sqrt{\beta}}\max_{\gamma\geq 0}\left (-\frac{e^{-2\gamma^2}}{\sqrt{2\pi}}+\gamma\mbox{erfc}\left (\frac{2\gamma}{\sqrt{2}}\right )-\frac{1}{4\gamma}\left (\mbox{erfc}\left (-\frac{2\gamma}{\sqrt{2}}\right )-1\right )-\gamma\beta \right ).\label{eq:thmwhfboxboxper}
\end{equation}
Let $\xi_{box}^{(l)}(\beta)$ be a scalar such that
\begin{equation}
(1-\epsilon_{1}^{(m)})\sqrt{\alpha f_{gar}\left (\frac{\kappa}{\sqrt{\beta}}\right )}+(1+\epsilon_{1}^{(n)} )\widehat{f_{box}}(\beta)-\epsilon_{5}^{(g)}>\frac{\xi_{box}^{(l)}(\beta)}{\sqrt{n}}.\label{eq:condxinthmstoc30boxper}
\end{equation}
Then
\begin{equation}
\hspace{-.3in} \lim_{n\rightarrow\infty}P(\xi_{box}(\beta)\geq \xi_{box}^{(l)}(\beta))=\lim_{n\rightarrow\infty}P\left (\min_{\x_i\in\left [-\frac{1}{\sqrt{\beta n}},\frac{1}{\sqrt{\beta n}}\right ],\|\x\|_2^2=1}\max_{\|\lambda\|_2=1,\lambda_i\geq 0}\left (\frac{\kappa}{\sqrt{\beta}}\lambda^T\1-\lambda^TH\x\right )\geq \xi_{box}^{(l)}(\beta)\right )\geq 1. \label{eq:probthmc3001per}
\end{equation}
Moreover, let $\xi_{box}^{(l)}$ be a scalar such that
\begin{equation}
\min_{\beta\in(0,1]}\left((1-\epsilon_{1}^{(m)})\sqrt{\alpha f_{gar}\left (\frac{\kappa}{\sqrt{\beta}}\right )}+(1+\epsilon_{1}^{(n)})\widehat{f_{box}}(\beta)-\epsilon_{5}^{(g)}\right )>\min_{\beta\in(0,1]}\frac{\xi_{box}^{(l)}(\beta)}{\sqrt{n}}
=\frac{\xi_{box}^{(l)}}{\sqrt{n}}.\label{eq:condxinthmstoc30boxpernobeta}
\end{equation}
Then
\begin{equation}
\lim_{n\rightarrow\infty}P(\xi_{box}\geq \xi_{box}^{(l)})=\lim_{n\rightarrow\infty}P\left (\min_{\x_i\in\left [-\frac{1}{\sqrt{n}},\frac{1}{\sqrt{n}}\right ],\|\x\|_2^2=1}\max_{\|\lambda\|_2=1,\lambda_i\geq 0}\left (\frac{\kappa}{\sqrt{\beta}}\lambda^T\1-\lambda^TH\x\right )\geq \xi_{box}^{(l)}\right )\geq 1 \label{eq:probthmc30boxpernobeta}
\end{equation}
and (\ref{eq:defprobucor2boxper}) is infeasible with overwhelming probability.
\label{thm:boxperc30}
\end{theorem}
\begin{proof}
Follows from the above discussion and the comments right after (\ref{eq:defprobucor2boxper}).
\end{proof}

In a more informal language (as earlier, essentially ignoring all technicalities and $\epsilon$'s) one has the following: let $\widehat{\alpha}$ be the smallest $\alpha$ such that
\begin{equation}
\min_{\beta\in(0,1]}\left(\sqrt{\alpha f_{gar}\left (\frac{\kappa}{\sqrt{\beta}}\right )}+\widehat{f_{box}}(\beta)\right )=0.\label{eq:infalphaboxper}
\end{equation}
Then as long as
\begin{equation}
\alpha>\widehat{\alpha},\label{eq:condalphaboxper}
\end{equation}
the problem in (\ref{eq:defprobucor2boxper}) will be infeasible with overwhelming probability. While it does take a bit of work to show that the above indeed matches the prediction obtained in \cite{GutSte90}, it is a straightforward functional analysis exercise and we omit it.

In the next subsection we will show that one can not really lower the storage capacity upper bound given above.

\subsubsection{Upper-bounding (the sign of) $\xi_{box}$}
\label{sec:uncorgardub}

In the previous subsection we designed a lower bound on $\xi_{box}$ which then helped us determine an upper bound on the critical storage capacity $\alpha_{box,c}$ of the box-constrained perceptron (essentially the one determined by Theorem \ref{thm:boxperc30}). In this subsection we will provide a mechanism that can be used to upper bound a quantity similar to $\xi_{box}$ (which will maintain the sign of $\xi_{box}$). Such an upper bound then can be used to obtain a lower bound on the critical storage capacity $\alpha_{box,c}$. As mentioned above, we will start by looking at a quantity very similar to $\xi_{box}$. In order to do that we will first recall on the definition of $\xi_{box}$ from (\ref{eq:feasboxper1}) and (\ref{eq:feasbox})
\begin{equation}
\xi_{box}=\min_{\beta\in(0,1]} \xi_{box}(\beta),\label{eq:feasboxper1ub}
\end{equation}
where
\begin{eqnarray}
\xi_{box}(\beta)=\min_{\x} \max_{\lambda\geq 0} & &  \kappa\lambda^T\1- \lambda^TH\x \nonumber \\
\mbox{subject to} & & \|\lambda\|_2= 1\nonumber \\
& & \x_i\in\left [-\frac{1}{\sqrt{n}},\frac{1}{\sqrt{n}}\right ],1\leq i\leq n \nonumber \\
& & \|\x\|_2^2=\beta .\label{eq:feasboxub}
\end{eqnarray}
The strategy presented above then assumed fixing a $\beta$ (from a discretized range of all $\beta$'s, namely $(0,1]$) and showing that for any such fixed $\beta$ $\xi_{box}>0$ with overwhelming probability. The bulk of the work then centered around determining conditions on $\alpha$ and $\kappa$ so that $\xi_{box}(\beta)>0$. Below we will design a similar mechanism that will be used to determine conditions on $\alpha$ and $\kappa$ so that $\xi_{box}(\beta)\leq 0$. In fact, instead of dealing explicitly with $\xi_{box}(\beta)$ defined above we will find a bit more convenient to deal with its a slight variation $\xi_{box,r}(\beta)$ which will be defined as
\begin{eqnarray}
\xi_{box,r}(\beta)=\min_{\x} \max_{\lambda\geq 0} & &  \kappa\lambda^T\1- \lambda^TH\x \nonumber \\
\mbox{subject to} & & \|\lambda\|_2\leq 1\nonumber \\
& & \x_i\in\left [-\frac{1}{\sqrt{n}},\frac{1}{\sqrt{n}}\right ],1\leq i\leq n \nonumber \\
& & \|\x\|_2^2\leq\beta .\label{eq:feasboxmodub}
\end{eqnarray}
Following further what was done in the previous section, one can scale down everything to obtain a redefined $\xi_{box,r}(\beta)$
\begin{eqnarray}
\xi_{box,r}(\beta)=\min_{\x} \max_{\lambda\geq 0} & &  \frac{\kappa}{\sqrt{\beta}}\lambda^T\1- \lambda^TH\x \nonumber \\
\mbox{subject to} & & \|\lambda\|_2\leq 1\nonumber \\
& & \x_i\in\left [-\frac{1}{\sqrt{\beta n}},\frac{1}{\sqrt{\beta n}}\right ],1\leq i\leq n \nonumber \\
& & \|\x\|_2^2\leq 1.\label{eq:feasbox1ub}
\end{eqnarray}
Using duality one has
\begin{eqnarray}
\xi_{box,r}= \max_{\lambda\geq 0} \min_{\x} & &  \frac{\kappa}{\sqrt{\beta}}\lambda^T\1- \lambda^T H\x \nonumber \\
\mbox{subject to}
& & \|\lambda\|_2\leq 1\nonumber \\
& & \x_i\in\left [-\frac{1}{\sqrt{\beta n}},\frac{1}{\sqrt{\beta n}}\right ],1\leq i\leq n \nonumber \\
& & \|\x\|_2\leq 1,\label{eq:uncormaxminboxperub}
\end{eqnarray}
and alternatively
\begin{eqnarray}
-\xi_{box,r}= \min_{\lambda\geq 0} \max_{\x} & &  -\frac{\kappa}{\sqrt{\beta}}\lambda^T\1+\lambda^T H\x \nonumber \\
\mbox{subject to}
& & \|\lambda\|_2\leq 1\nonumber \\
& & \x_i\in\left [-\frac{1}{\sqrt{\beta n}},\frac{1}{\sqrt{\beta n}}\right ],1\leq i\leq n \nonumber \\
& & \|\x\|_2\leq 1.\label{eq:uncormaxmin1boxperub}
\end{eqnarray}
We will now proceed in a fashion similar to the one presented in the previous subsection. We will make use of the following lemma (the lemma is fairly similar to Lemmas \ref{lemma:negproblemma}, \ref{lemma:negproblemma01per}, \ref{lemma:negproblemmaboxper}, and of course fairly similar to Lemma 3.1 in \cite{Gordon88}).
\begin{lemma}
Let $H$ be an $m\times n$ matrix with i.i.d. standard normal components. Let $\g$ and $\h$ be $m\times 1$ and $n\times 1$ vectors, respectively, with i.i.d. standard normal components. Also, let $g$ be a standard normal random variable and let $\zeta_{\lambda}$ be a function of $\x$. Then
\begin{multline}
P (\min_{\|\lambda\|_2\leq 1,\lambda_i\geq 0}\max_{\|\x\|_2\leq 1,\x_i\in\left [-\frac{1}{\sqrt{\beta n}},\frac{1}{\sqrt{\beta n}}\right ]}(\lambda^T H\x+g\|\lambda\|_2\|\x\|_2-\zeta_{\lambda} )\geq 0 )\\\geq
P(\min_{\|\lambda\|_2\leq 1,\lambda_i\geq 0}\max_{\|\x\|_2\leq 1,\x_i\in\left [-\frac{1}{\sqrt{\beta n}},\frac{1}{\sqrt{\beta n}}\right ]} (\|\x\|_2\g^T\lambda+\|\lambda\|_2\h^T\x-\zeta_{\lambda})\geq 0 ).\label{eq:negproblemmaboxperub}
\end{multline}\label{lemma:negproblemmaboxperub}
\end{lemma}
\begin{proof}
The discussion related to the proof of Lemma \ref{lemma:negproblemma} applies here as well.
\end{proof}

Let $\zeta_{\lambda}=\frac{\kappa}{\sqrt{\beta}}\lambda^T\1+\epsilon_{5}^{(g)}\sqrt{n}\|\lambda\|_2\|\x\|_2$ with $\epsilon_{5}^{(g)}>0$ being an arbitrarily small constant independent of $n$. We will follow the strategy of the previous subsection and start by first looking at the right-hand side of the inequality in (\ref{eq:negproblemmaboxperub}). The following is then the probability of interest
\begin{equation}
P\left (\min_{\|\lambda\|_2\leq 1,\lambda_i\geq 0,\lambda\neq 0}\max_{\|\x\|_2\leq 1,\x_i\in\left [-\frac{1}{\sqrt{\beta n}},\frac{1}{\sqrt{\beta n}}\right ]}\left (\|\x\|_2\g^T\lambda+\|\lambda\|_2\h^T\x-\frac{\kappa}{\sqrt{\beta}}\lambda^T\1-\epsilon_{5}^{(g)}\sqrt{n}\|\lambda\|_2\|\x\|_2\right )>0\right ),\label{eq:negprobanal0boxperub}
\end{equation}
where for the easiness of writing we removed possibility $\lambda=0$ (also, such a case contributes in no way to the possibility that $-\xi_{box,r}(\beta)<0)$.
After solving the maximization over $\x$ one obtains
\begin{multline}
P\left (\min_{\|\lambda\|_2\leq 1,\lambda_i\geq 0,\lambda\neq 0}\max_{\|\x\|_2\leq 1,\x_i\in\left [-\frac{1}{\sqrt{\beta n}},\frac{1}{\sqrt{\beta n}}\right ]}\left (\|\x\|_2\g^T\lambda+\|\lambda\|_2\h^T\x-\frac{\kappa}{\sqrt{\beta}}\lambda^T\1-\epsilon_{5}^{(g)}\sqrt{n}\|\lambda\|_2\|\x\|_2\right )> 0\right )\\=P\left (\min_{\|\lambda\|_2\leq 1,\lambda_i\geq 0,\lambda\neq 0}\left (\max\left (0,-f_{box}^{(r)}(\h,\beta)\|\lambda\|_2+\g\lambda-\epsilon_{5}^{(g)}\sqrt{n}\|\lambda\|_2\right )-\frac{\kappa}{\sqrt{\beta}}\lambda^T\1\right )>0\right ).\label{eq:negprobanal1boxperub}
\end{multline}
Now, we will for a moment assume that $\alpha$ (i.e., $m$ and $n$) and $\frac{\kappa}{\sqrt{\beta}}$ are such that
\begin{equation}
\lim_{n\rightarrow\infty}P\left (\min_{\|\lambda\|_2\leq 1,\lambda_i\geq 0,\lambda\neq 0}\left (-f_{box}^{(r)}(\h,\beta)\|\lambda\|_2+\g\lambda-\epsilon_{5}^{(g)}\sqrt{n}\|\lambda\|_2-\frac{\kappa}{\sqrt{\beta}}\lambda^T\1\right )> 0\right )=1.\label{eq:negprobanal2boxperub}
\end{equation}
That would also imply that
\begin{equation}
\lim_{n\rightarrow\infty}P\left (\min_{\|\lambda\|_2\leq 1,\lambda_i\geq 0,\lambda\neq 0}\left (\max(0,-f_{box}^{(r)}(\h,\beta)\|\lambda\|_2+\g\lambda-\epsilon_{5}^{(g)}\sqrt{n}\|\lambda\|_2)-\frac{\kappa}{\sqrt{\beta}}\lambda^T\1\right )>0\right )=1.\label{eq:negprobanal3boxperub}
\end{equation}
What is then left to be done is to determine an $\alpha=\frac{m}{n}$ such that (\ref{eq:negprobanal2boxperub}) holds. One then easily has
\begin{multline}
P\left (\min_{\|\lambda\|_2\leq 1,\lambda_i\geq 0,\lambda\neq 0}\left (-f_{box}^{(r)}(\h,\beta)\|\lambda\|_2+\g\lambda-\epsilon_{5}^{(g)}\sqrt{n}\|\lambda\|_2-\frac{\kappa}{\sqrt{\beta}}\lambda^T\1\right )> 0\right )\\=
P\left (\min_{\|\lambda\|_2\leq 1,\lambda_i\geq 0,\lambda\neq 0}\|\lambda\|_2\left (-f_{box}^{(r)}(\h,\beta)-\|\left (\g-\frac{\kappa}{\sqrt{\beta}}\1\right )_-\|_2-\epsilon_{5}^{(g)}\sqrt{n}\right )> 0\right ),\label{eq:negprobanal4boxperub}
\end{multline}
where similarly to what we had in Section \ref{sec:probanalrigboxper} $\left (\g-\frac{\kappa}{\sqrt{\beta}}\1\right )_-$ is $\left (\g-\frac{\kappa}{\sqrt{\beta}}\1\right )$ vector with positive components replaced by zeros. Also similarly to what we did in Section \ref{sec:probanalrigboxper}, since $\h$ is a vector of $n$ i.i.d. standard normal variables one can write
\begin{equation}
P(-f_{box}^{(r)}(\h,\beta)>-(1-\epsilon_{1}^{(n)})f_{box}(\beta)\sqrt{n})\geq 1-e^{-\epsilon_{2}^{(n)} n},\label{eq:devhboxperub}
\end{equation}
where we recall that $f_{box}(\beta)$ is as in (\ref{eq:fbox}), i.e.
\begin{equation}
f_{box}(\beta)=\lim_{n\rightarrow \infty}\frac{Ef_{box}^{(r)}(\h,\beta)}{\sqrt{n}}=\lim_{n\rightarrow\infty}\frac{E\left ( \min_{\x_i\in\left [-\frac{1}{\sqrt{\beta n}},\frac{1}{\sqrt{\beta n}}\right ],\|\x\|_2^2=1}\h^T\x\right )}{\sqrt{n}},\label{eq:fboxub}
\end{equation}
and $\epsilon$'s are as described in Section \ref{sec:probanalrigboxper}. Along the same lines, since $\g$ is a vector of $m$ i.i.d. standard normal variables one has similarly to what was done in previous sections (and ultimately in \cite{StojnicGardGen13})
\begin{equation}
P\left (\sqrt{\sum_{i=1}^{n}\left (\min\left \{\g_i-\frac{\kappa}{\sqrt{\beta}},0\right \}\right )^2}<(1+\epsilon_{1}^{(m)})\sqrt{mf_{gar}\left (\frac{\kappa}{\sqrt{\beta}}\right )}\right )\geq 1-e^{-\epsilon_{2}^{(m)} m},\label{eq:devgboxperub}
\end{equation}
where we recall that $\epsilon_{1}^{(m)}>0$ is an arbitrarily small constant and $\epsilon_{2}^{(m)}$ is a constant dependent on $\epsilon_{1}^{(m)}$ and
$f_{gar}(\frac{\kappa}{\sqrt{\beta}})$ but independent of $n$. Then a combination of (\ref{eq:negprobanal4boxperub}), (\ref{eq:devhboxperub}), and (\ref{eq:devgboxperub}) gives
\begin{multline}
P\left (\min_{\|\lambda\|_2\leq 1,\lambda_i\geq 0,\lambda\neq 0}\left (-f_{box}^{(r)}(\h,\beta)\|\lambda\|_2+\g\lambda-\epsilon_{5}^{(g)}\sqrt{n}\|\lambda\|_2-\frac{\kappa}{\sqrt{\beta}}\lambda^T\1\right )> 0\right )\\\geq
(1-e^{-\epsilon_{2}^{(m)} m})(1-e^{-\epsilon_{2}^{(n)} n})
P\left ((1-\epsilon_{1}^{(n)})(-f_{box}(\beta))\sqrt{n}-(1+\epsilon_{1}^{(m)})\sqrt{mf_{gar}\left (\frac{\kappa}{\sqrt{\beta}}\right )}-\epsilon_{5}^{(g)}\sqrt{n}> 0\right ).
\label{eq:negprobanal22boxperub}
\end{multline}
If
\begin{eqnarray}
& & (1-\epsilon_{1}^{(n)})(-f_{box}(\beta))\sqrt{n}-(1+\epsilon_{1}^{(m)})\sqrt{mf_{gar}\left (\frac{\kappa}{\sqrt{\beta}}\right )}-\epsilon_{5}^{(g)}\sqrt{n}>0\nonumber \\
& \Leftrightarrow & -(1-\epsilon_{1}^{(n)})f_{box}(\beta)-(1+\epsilon_{1}^{(m)})\sqrt{\alpha f_{gar}\left (\frac{\kappa}{\sqrt{\beta}}\right )}-\epsilon_{5}^{(g)}>0,\label{eq:negcondxipuboxperub}
\end{eqnarray}
one then has from (\ref{eq:negprobanal22boxperub})
\begin{equation}
\lim_{n\rightarrow\infty}P\left (\min_{\|\lambda\|_2\leq 1,\lambda_i\geq 0,\lambda\neq 0}\left (-f_{box}^{(r)}(\h,\beta)\|\lambda\|_2+\g\lambda-\epsilon_{5}^{(g)}\sqrt{n}\|\lambda\|_2-\frac{\kappa}{\sqrt{\beta}}\lambda^T\1\right )> 0\right )\geq 1.\label{eq:negprobanal33boxperub}
\end{equation}
A combination of  (\ref{eq:negprobanal1boxperub}),  (\ref{eq:negprobanal2boxperub}),  (\ref{eq:negprobanal3boxperub}), and  (\ref{eq:negprobanal33boxperub}) gives that if (\ref{eq:negcondxipuboxperub}) holds then
\begin{equation}
\hspace{-.3in}\lim_{n\rightarrow\infty}P\left (\min_{\|\lambda\|_2\leq 1,\lambda_i\geq 0,\lambda\neq 0}\max_{\|\x\|_2\leq 1,\x_i\in\left [-\frac{1}{\sqrt{\beta n}},\frac{1}{\sqrt{\beta n}}\right ]}\left (\|\x\|_2\g^T\lambda+\|\lambda\|_2\h^T\x-\frac{\kappa}{\sqrt{\beta}}\lambda^T\1-\epsilon_{5}^{(g)}\sqrt{n}\|\lambda\|_2\|\x\|_2\right )> 0\right )\geq 1.\label{eq:negprobanal44boxperub}
\end{equation}

We will now look at the left-hand side of the inequality in (\ref{eq:negproblemmaboxperub}). The following is then the probability of interest
\begin{equation}
P\left (\min_{\|\lambda\|_2\leq 1,\lambda_i\geq 0}\max_{\|\x\|_2\leq 1,\x_i\in\left [-\frac{1}{\sqrt{\beta n}},\frac{1}{\sqrt{\beta n}}\right ]}\left (\lambda^TH\x-\frac{\kappa}{\sqrt{\beta}}\lambda^T\1+(g-\epsilon_{5}^{(g)}\sqrt{n})\|\lambda\|_2\|\x\|_2\right )\geq 0\right ).\label{eq:leftnegprobanal0boxperub}
\end{equation}
Since $P(g\geq\epsilon_{5}^{(g)}\sqrt{n})<e^{-\epsilon_{6}^{(g)} n}$ (where $\epsilon_{6}^{(g)}$ is, as all other $\epsilon$'s in this paper are, independent of $n$) from (\ref{eq:leftnegprobanal0boxperub}) we have
\begin{multline}
P\left (\min_{\|\lambda\|_2\leq 1,\lambda_i\geq 0}\max_{\|\x\|_2\leq 1,\x_i\in\left [-\frac{1}{\sqrt{\beta n}},\frac{1}{\sqrt{\beta n}}\right ]}\left (\lambda^TH\x-\frac{\kappa}{\sqrt{\beta}}\lambda^T\1+(g-\epsilon_{5}^{(g)}\sqrt{n})\|\lambda\|_2\|\x\|_2\right )\geq 0\right )\\\leq P\left (\min_{\|\lambda\|_2\leq 1,\lambda_i\geq 0}\max_{\|\x\|_2\leq 1,\x_i\in\left [-\frac{1}{\sqrt{\beta n}},\frac{1}{\sqrt{\beta n}}\right ]}\left (\lambda^TH\x-\frac{\kappa}{\sqrt{\beta}}\lambda^T\1\right )\geq 0\right )+e^{-\epsilon_{6}^{(g)} n}.\label{eq:leftnegprobanal1boxperub}
\end{multline}
When $n$ is large from (\ref{eq:leftnegprobanal1boxperub}) we then have
\begin{multline}
\hspace{-.7in}\lim_{n\rightarrow \infty}P\left (\min_{\|\lambda\|_2\leq 1,\lambda_i\geq 0}\max_{\|\x\|_2\leq 1,\x_i\in\left [-\frac{1}{\sqrt{\beta n}},\frac{1}{\sqrt{\beta n}}\right ]}\left (\lambda^TH\x-\frac{\kappa}{\sqrt{\beta}}\lambda^T\1+(g-\epsilon_{5}^{(g)}\sqrt{n})\|\lambda\|_2\|\x\|_2\right )\geq 0\right )\\\leq \lim_{n\rightarrow \infty}P\left (\min_{\|\lambda\|_2\leq 1,\lambda_i\geq 0}\max_{\|\x\|_2\leq 1,\x_i\in\left [-\frac{1}{\sqrt{\beta n}},\frac{1}{\sqrt{\beta n}}\right ]}\left (\lambda^TH\x-\frac{\kappa}{\sqrt{\beta}}\lambda^T\1\right )\geq 0\right ).\label{eq:leftnegprobanal2boxperub}
\end{multline}
Assuming that (\ref{eq:negcondxipuboxperub}) holds, then a combination of (\ref{eq:uncormaxmin1boxperub}), (\ref{eq:negproblemmaboxperub}), (\ref{eq:negprobanal44boxperub}), and (\ref{eq:leftnegprobanal2boxperub}) gives
\begin{multline}
\lim_{n\rightarrow \infty}P(\xi_{box,r}(\beta)\leq 0)  =  \lim_{n\rightarrow \infty}P(-\xi_{box,r}(\beta)\geq 0) \\
= \lim_{n\rightarrow \infty}P\left (\min_{\|\lambda\|_2\leq 1,\lambda_i\geq 0}\max_{\|\x\|_2\leq 1,\x_i\in\left [-\frac{1}{\sqrt{\beta n}},\frac{1}{\sqrt{\beta n}}\right ]}\left (\lambda^TH\x-\frac{\kappa}{\sqrt{\beta}}\lambda^T\1\right )\geq 0\right )
\\
 \geq
\lim_{n\rightarrow \infty}P\left (\min_{\|\lambda\|_2\leq 1,\lambda_i\geq 0}\max_{\|\x\|_2\leq 1,\x_i\in\left [-\frac{1}{\sqrt{\beta n}},\frac{1}{\sqrt{\beta n}}\right ]}\left (\lambda^TH\x-\frac{\kappa}{\sqrt{\beta}}\lambda^T\1+(g-\epsilon_{5}^{(g)}\sqrt{n})\|\lambda\|_2\|\x\|_2\right )\geq 0\right )\\
\hspace{-.4in} \geq
\lim_{n\rightarrow\infty}P\left (\min_{\|\lambda\|_2\leq 1,\lambda_i\geq 0,\lambda\neq 0}\max_{\|\x\|_2\leq 1,\x_i\in\left [-\frac{1}{\sqrt{\beta n}},\frac{1}{\sqrt{\beta n}}\right ]}\left (\|\x\|_2\g^T\lambda+\|\lambda\|_2\h^T\x-\frac{\kappa}{\sqrt{\beta}}\lambda^T\1-\epsilon_{5}^{(g)}\sqrt{n}\|\lambda\|_2\|\x\|_2\right )> 0\right )
 \geq  1.\\\label{eq:leftnegprobanal3boxperub}
\end{multline}
From (\ref{eq:leftnegprobanal3boxperub}) one then has
\begin{equation}
\lim_{n\rightarrow \infty}P(\xi_{box,r}(\beta)> 0)=1-\lim_{n\rightarrow \infty}P(\xi_{box,r}(\beta)\leq 0)\leq 0,\label{eq:leftnegprobanal4boxperub}
\end{equation}
which implies that if (\ref{eq:negcondxipuboxperub}) holds then (\ref{eq:defprobucor2boxper}) is feasible with overwhelming probability.

We summarize our results from this subsection in the following theorem.

\begin{theorem}
Let $H$ be an $m\times n$ matrix with i.i.d. standard normal components. Let $n$ be large and let $m=\alpha n$, where $\alpha>0$ is a constant independent of $n$. Let $\xi_{box}$ be as in (\ref{eq:feasboxper}) and let $\frac{\kappa}{\sqrt{\beta}}>0$ be a scalar constant independent of $n$. Let all $\epsilon$'s be arbitrarily small constants independent of $n$. Further, let $\g_i$ be a standard normal random variable and set
\begin{equation}
\hspace{-.65in}f_{gar}\left (\frac{\kappa}{\sqrt{\beta}}\right )=\frac{1}{\sqrt{2\pi}}\int_{-\frac{\kappa}{\sqrt{\beta}}}^{\infty}\left (\g_i+\frac{\kappa}{\sqrt{\beta}}\right )^2e^{-\frac{\g_i^2}{2}}d\g_i
=\frac{1}{\sqrt{2\pi}}\int_{-\infty}^{\frac{\kappa}{\sqrt{\beta}}}\left (\g_i-\frac{\kappa}{\sqrt{\beta}}\right )^2e^{-\frac{\g_i^2}{2}}d\g_i=\frac{\kappa e^{-\frac{\kappa^2}{2\beta}}}{\sqrt{2\beta\pi}}+\frac{(\frac{\kappa^2}{\beta}+1)\mbox{erfc}\left ( -\frac{\kappa}{\sqrt{2\beta}}\right )}{2}\label{eq:thmfgarboxperub}
\end{equation}
and
\begin{equation}
\widehat{f_{box}}(\beta)=\frac{1}{\sqrt{\beta}}\max_{\gamma\geq 0}\left (-\frac{e^{-2\gamma^2}}{\sqrt{2\pi}}+\gamma\mbox{erfc}\left (\frac{2\gamma}{\sqrt{2}}\right )-\frac{1}{4\gamma}\left (\mbox{erfc}\left (-\frac{2\gamma}{\sqrt{2}}\right )-1\right )-\gamma\beta \right ).\label{eq:thmwhfboxboxper}
\end{equation}
Let $\alpha$ be a scalar such that
\begin{equation}
-(1-\epsilon_{1}^{(n)} )\widehat{f_{box}}(\beta)-(1+\epsilon_{1}^{(m)})\sqrt{\alpha f_{gar}\left (\frac{\kappa}{\sqrt{\beta}}\right )}+\epsilon_{5}^{(g)}>0.\label{eq:condxinthmstoc30boxperub}
\end{equation}
Then
\begin{equation}
\lim_{n\rightarrow \infty}P(\xi_{box,r}(\beta)> 0)=1-\lim_{n\rightarrow \infty}P(\xi_{box,r}(\beta)\leq 0)\leq 0. \label{eq:probthmc30boxperub}
\end{equation}
Moreover, let $\alpha$ be a scalar such that
\begin{equation}
\max_{\beta\in(0,1]}\left(-(1-\epsilon_{1}^{(n)})\widehat{f_{box}}(\beta)-(1+\epsilon_{1}^{(m)})\sqrt{\alpha f_{gar}\left (\frac{\kappa}{\sqrt{\beta}}\right )}+\epsilon_{5}^{(g)}\right )>0.\label{eq:condxinthmstoc30boxpernobetaub}
\end{equation}
Then
\begin{equation}
\hspace{-.3in}\lim_{n\rightarrow\infty}P(\xi_{box,r}=\min_{\beta\in(0,1]}\xi_{box,r}(\beta)> 0)=\lim_{n\rightarrow\infty}P\left (\min_{\x_i\in\left [-\frac{1}{\sqrt{n}},\frac{1}{\sqrt{n}}\right ],\|\x\|_2^2=1}\max_{\|\lambda\|_2=1,\lambda_i\geq 0}\left (\frac{\kappa}{\sqrt{\beta}}\lambda^T\1-\lambda^TH\x\right )> 0\right )\leq 0 \label{eq:probthmc30boxpernobetaub}
\end{equation}
and (\ref{eq:defprobucor2boxper}) is feasible with overwhelming probability.
\label{thm:boxperc30ub}
\end{theorem}
\begin{proof}
Follows from the above discussion, comments right after (\ref{eq:defprobucor2boxper}), and the recognition that $\xi_{box,r}$ and $\xi_{box}$ have the same sign.
\end{proof}

Similarly to what was done in Section \ref{sec:probanalrigboxper}, one can again be a bit more informal and ignore all technicalities and $\epsilon$'s. After doing so one has the following: let $\widehat{\alpha}$ be the smallest $\alpha$ such that
\begin{equation}
\min_{\beta\in(0,1]}\left(\sqrt{\alpha f_{gar}\left (\frac{\kappa}{\sqrt{\beta}}\right )}+\widehat{f_{box}}(\beta)\right )=0.\label{eq:infalphaboxperub}
\end{equation}
Then as long as
\begin{equation}
\alpha<\widehat{\alpha},\label{eq:condalphaboxperub}
\end{equation}
the problem in (\ref{eq:defprobucor2boxper}) will be feasible with overwhelming probability. As mentioned in Section \ref{sec:probanalrigboxper} the above condition matches the one obtained in \cite{GutSte90} based on a replica statistical mechanics type of approach.

The results obtained based on Theorems \ref{thm:boxperc30} and \ref{thm:boxperc30ub} (as well as those predicted assuming replica symmetry and given in \cite{GutSte90}) are presented in Figure \ref{fig:boxperc30c3opt}. For the values of $\alpha$ that are to the right of the given curve the memory will not operate correctly with overwhelming probability. On the other hand, for the values of $\alpha$ that are to the left of the given curve the memory will operate correctly with overwhelming probability. This of course follows from the fact that with overwhelming probability over $H$ the inequalities in (\ref{eq:defprobucor2boxper}) will (or will not) be simultaneously satisfiable.
\begin{figure}[htb]
\centering
\centerline{\epsfig{figure=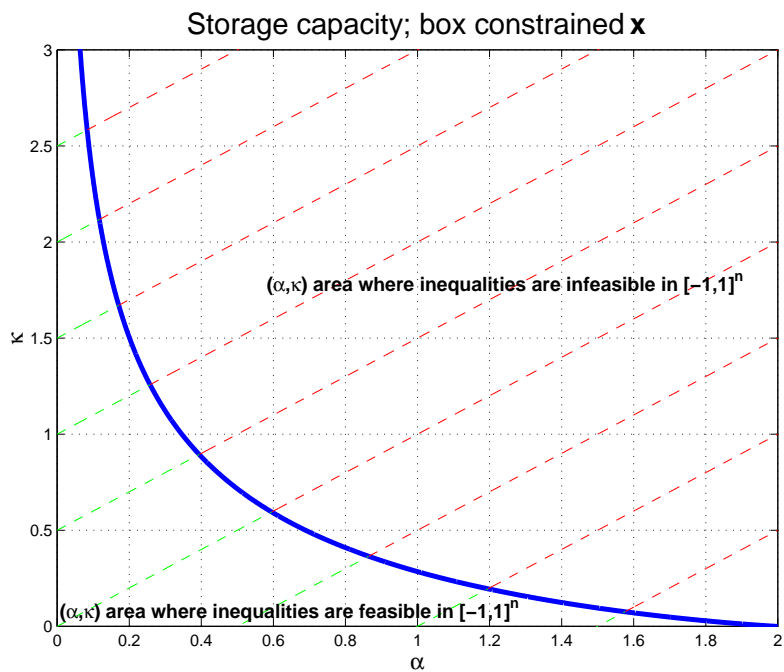,width=10.5cm,height=9cm}}
\caption{$\kappa$ as a function of $\alpha$; $\x\in \left [-\frac{1}{\sqrt{n}},\frac{1}{\sqrt{n}} \right ]^n$}
\label{fig:boxperc30c3opt}
\end{figure}

\section{Conclusion}
\label{sec:conc}

In this paper we looked at a special class of perceptrons, that we called discrete preceptrons. While various features are of interest in studying pretty much any type of perceptron we here focused on its properties when used as storage memories. More specifically, we considered several mathematical problems that eventually correspond to computing what is in neural networks terminology known as the storage capacity of perceptrons.

We considered two special classes of discrete perceptrons: one that we called $\pm 1$ perceptrons and another that we called $0/1$ perceptrons. For both of these classes we, in a statistical context determined the upper bounds on their storage capacities. Moreover, these happen to match the predictions obtained within the neural networks framework through the use of replica symmetry theory from statistical mechanics. In addition to two these two classes, we also consider a continuous type of perceptron that we referred to as the box-constrained perceptron. These perceptrons can be viewed as a limiting version of the so-called digital perceptrons which on the other hand are an extension of the binary or $\pm 1$ perceptrons. For the box-constrained perceptrons we determined the exact value of the storage capacity. This of course confirmed earlier predictions obtained through the replica symmetry type of approach of statistical mechanics.

Of course, it is a no surprise that for the box-constrained case we obtained the exact values of the optimal storage capacity. Since computing these capacities amounts to solving an optimization problem which turns out to be doable in a reasonable (actually polynomial) amount of time the results we obtained here are then in a complete agreement with what the theory that we developed in \cite{StojnicRegRndDlt10} predicts. Of course, this (and ultimately the entire theory we developed in \cite{StojnicRegRndDlt10}) also provides a rigorous mathematical confirmation for long established beliefs of physicists.

As for the results that we presented for purely discrete perceptrons, we presented essentially a powerful mechanism that can be used to obtain the upper bounds on their storage capacities. Moreover, for the $\pm 1$ perceptron we then introduced a modification of the mechanism that can lower these upper bounds. While doing so, we also uncovered an intersting phenomenon that happens in the analysis of $\pm 1$ perceptrons. Namely, the lowered upper bounds happen to match the simple combinatorial bounds that have long served as a clear mathematical proof that the replica symmetry results are well above the true storage capacity values.

We should also mention that besides the storage capacities many other features of perceptrons are also of interest. Some of them also relate to their memory capacities while others relate to functioning of these memories. The concepts that we presented can be utilized to characterize many of these features and we will present results in these directions elsewhere. Also, the results we presented relate to a particular statistical version of the spherical perceptron. Such a version is within the frame of neural networks/statistical mechanics typically called uncorrelated. As was the case with the results we presented in \cite{StojnicGardGen13} when we studied the basics of the spherical perceptrons, the results we presented here can also be translated to cover the corresponding correlated case. While on the topic of randomness, we should emphasize that strictly speaking we instead of typical binary patterns assumed standard normal ones. This was to done to make the presentation as easy as possible. As mentioned earlier in the paper (and as discussed to a much greater detail in \cite{StojnicHopBnds10,StojnicMoreSophHopBnds10}), all results that we presented easily extend beyond the standard Gaussian setup we utilized. A way to show that would be to utilize a repetitive use of the central limit theorem. For example, a particularly simple and elegant approach in that direction would be the one of Lindeberg \cite{Lindeberg22}. Adapting our exposition to fit into the framework of the Lindeberg principle is relatively easy and in fact if one uses the elegant approach of \cite{Chatterjee06} pretty much a routine. However, as we mentioned when studying the Hopfield and Little models \cite{StojnicHopBnds10,StojnicMoreSophHopBnds10,StojnicAsymmLittBnds11}, since we did not create these techniques we chose not to do these routine generalizations.

In this paper we primarily focused on the behavior of the storage capacity when viewed from an analytical point of view. In other words, we focused on quantifying analytically what the capacity would be in a statistical scenario. Of course, a tone of interesting questions related to this same problem arise if one looks at it from an algorithmic point of view. For example, one may wonder how easy is to actually determine the strengths of the bonds that do achieve the storage capacity (or to be more in alignment with what we proved here, a lower bound of the storage capacity). While this problem is relatively easy (in fact, as mentioned above, solvable in polynomial time) for the box-constrained perceptrons, it is much harder for the purely discrete counterparts we studied here. In this paper we were mostly concerned with certain analytical properties of the discrete perceptrons and consequently did not present any considerations in the algorithmic direction. However, we do mention that one can design algorithms similar to those designed for problems considered in \cite{StojnicUpperSec13}. Since an algorithmic consideration of discrete perceptrons is an important topic on its own, we will present a more detailed discussion in this direction in a separate paper.

Also, we emphasized on multiple occasions throughout the paper that here we considered only three particular versions of discrete perceptrons (in fact one of them as a limiting version essentially becomes continuous). Also as we mentioned throughout the paper, we did so to enable an easy flowing exposition and to avoid overloading the presentation of the main concepts with unnecessary details of different perceptron versions. However, we should add
that many other discrete versions are of interest and in fact have been studied analytically or even algorithmically throughout the vast literature related to perceptrons. All concepts that we presented here can be easily adapted to pretty much any of these versions. That typically does take some work but is in principle a routine and we will present some of concrete results in these directions elsewhere.

\begin{singlespace}
\bibliographystyle{plain}
\bibliography{GardDiscPerRefs}

\begin{thebibliography}{10}

\bibitem{AgiAnnBarCooTan13b}
E.~Agliari, A.~Annibale, A.~Barra, A.C.C. Coolen, and D.~Tantari.
\newblock Immune networks: multi-tasking capabilities at medium load.
\newblock 2013.
\newblock avaialable at arxiv.

\bibitem{AgiAnnBarCooTan13a}
E.~Agliari, A.~Annibale, A.~Barra, A.C.C. Coolen, and D.~Tantari.
\newblock Retrieving infinite numbers of patterns in a spin-glass model of
  immune networks.
\newblock 2013.
\newblock avaialable at arxiv.

\bibitem{AgiAstBarBurUgu12}
E.~Agliari, L.~Asti, A.~Barra, R.~Burioni, and G.~Uguzzoni.
\newblock Analogue neural networks on correlated random graphs.
\newblock {\em J. Phys. A: Math. Theor.}, 45:365001, 2012.

\bibitem{AgiBarBarGalGueMoa12}
E.~Agliari, A.~Barra, Silvia Bartolucci, A.~Galluzzi, F.~Guerra, and F.~Moauro.
\newblock Parallel processing in immune networks.
\newblock {\em Phys. Rev. E}, 2012.

\bibitem{AgiBarGalGueMoa12}
E.~Agliari, A.~Barra, A.~Galluzzi, F.~Guerra, and F.~Moauro.
\newblock Multitasking associative networks.
\newblock {\em Phys. Rev. Lett}, 2012.

\bibitem{BalVen87}
P.~Baldi and S.~Venkatesh.
\newblock Number od stable points for spin-glasses and neural networks of
  higher orders.
\newblock {\em Phys. Rev. Letters}, 58(9):913--916, Mar. 1987.

\bibitem{BruParRit92}
R.~Brunetti, G.~Parisi, and F.~Ritort.
\newblock Asymmetric little spin glas model.
\newblock {\em Physical Review B}, 46(9), September 1992.

\bibitem{Cameron60}
S.~H. Cameron.
\newblock Tech-report 60-600.
\newblock {\em Proceedings of the bionics symposium}, pages 197--212, 1960.
\newblock Wright air development division, {D}ayton, {O}hio.

\bibitem{Chatterjee06}
S.~Chatterjee.
\newblock A generalization of the {L}indenberg principle.
\newblock {\em The Annals of Probability}, 34(6):2061--2076.

\bibitem{Cover65}
T.~Cover.
\newblock Geomretrical and statistical properties of systems of linear
  inequalities with applications in pattern recognition.
\newblock {\em IEEE Transactions on Electronic Computers}, (EC-14):326--334,
  1965.

\bibitem{Gar88}
E.~Gardner.
\newblock The space of interactions in neural networks models.
\newblock {\em J. Phys. A: Math. Gen.}, 21:257--270, 1988.

\bibitem{GarDer88}
E.~Gardner and B.~Derrida.
\newblock Optimal storage properties of neural networks models.
\newblock {\em J. Phys. A: Math. Gen.}, 21:271--284, 1988.

\bibitem{Gordon85}
Y.~Gordon.
\newblock Some inequalities for gaussian processes and applications.
\newblock {\em Israel Journal of Mathematics}, 50(4):265--289, 1985.

\bibitem{Gordon88}
Y.~Gordon.
\newblock On {M}ilman's inequality and random subspaces which escape through a
  mesh in ${R}^n$.
\newblock {\em Geometric Aspect of of functional analysis, Isr. Semin. 1986-87,
  Lect. Notes Math}, 1317, 1988.

\bibitem{GutSte90}
H.~Gutfreund and Y.~Stein.
\newblock Capacity of neural networks with discrete synaptic couplings.
\newblock {\em J. Physics A: Math. Gen}, 23:2613, 1990.

\bibitem{Joseph60}
R.~D. Joseph.
\newblock The number of orthants in $n$-space instersected by an
  $s$-dimensional subspace.
\newblock {\em Tech. memo 8, project {PARA}}, 1960.
\newblock Cornel aeronautical lab., Buffalo, N.Y.

\bibitem{KimRoc98}
J.~H. Kim and J.~R. Roche.
\newblock Covering cubes by random half cubes with applications to biniary
  neural networks.
\newblock {\em Journal of Computer and System Sciences}, 56:223--252, 1998.

\bibitem{KraMez89}
W.~Krauth and M.~Mezard.
\newblock Storage capacity of memory networks with binary couplings.
\newblock {\em J. Phys. France}, 50:3057--3066, 1989.

\bibitem{Lindeberg22}
J.~W. Lindeberg.
\newblock Eine neue herleitung des exponentialgesetzes in der
  wahrscheinlichkeitsrechnung.
\newblock {\em Math. Z.}, 15:211--225, 1922.

\bibitem{Winder}
R.~O.Winder.
\newblock {\em Threshold logic}.
\newblock Ph. D. dissertation, Princetoin University, 1962.

\bibitem{Schlafli}
L.~Schlafli.
\newblock {\em Gesammelte Mathematische AbhandLungen I}.
\newblock Basel, Switzerland: Verlag Birkhauser, 1950.

\bibitem{SchTir02}
M.~Shcherbina and Brunello Tirozzi.
\newblock On the volume of the intrersection of a sphere with random half
  spaces.
\newblock {\em C. R. Acad. Sci. Paris. Ser I}, (334):803--806, 2002.

\bibitem{SchTir03}
M.~Shcherbina and Brunello Tirozzi.
\newblock Rigorous solution of the {G}ardner problem.
\newblock {\em Comm. on Math. Physiscs}, (234):383--422, 2003.

\bibitem{StojnicGardGen13}
M.~Stojnic.
\newblock Another look at the {G}ardner problem.
\newblock available at arXiv.

\bibitem{StojnicAsymmLittBnds11}
M.~Stojnic.
\newblock Asymmetric {L}ittle model and its ground state energies.
\newblock available at arXiv.

\bibitem{StojnicHopBnds10}
M.~Stojnic.
\newblock Bounding ground state energy of {H}opfield models.
\newblock available at arXiv.

\bibitem{StojnicMoreSophHopBnds10}
M.~Stojnic.
\newblock Lifting/lowering {H}opfield models ground state energies.
\newblock available at arXiv.

\bibitem{StojnicGorEx10}
M.~Stojnic.
\newblock Meshes that trap random subspaces.
\newblock available at arXiv.

\bibitem{StojnicGardSphNeg13}
M.~Stojnic.
\newblock Negative spherical perceptron.
\newblock available at arXiv.

\bibitem{StojnicRegRndDlt10}
M.~Stojnic.
\newblock Regularly random duality.
\newblock available at arXiv.

\bibitem{StojnicGardSphErr13}
M.~Stojnic.
\newblock Spherical perceptron as a storage memory with limited errors.
\newblock available at arXiv.

\bibitem{StojnicUpperSec13}
M.~Stojnic.
\newblock Upper-bounding $\ell_1$-optimization sectional thresholds.
\newblock available at arXiv.

\bibitem{StojnicCSetam09}
M.~Stojnic.
\newblock Various thresholds for $\ell_1$-optimization in compressed sensing.
\newblock {\em submitted to IEEE Trans. on Information Theory}, 2009.
\newblock available at arXiv:0907.3666.

\bibitem{TalBook}
M.~Talagrand.
\newblock {\em Mean field models for spin glasses}.
\newblock A series of modern surveys in mathematics 54, Springer-Verlag, Berlin
  Heidelberg, 2011.

\bibitem{Ven86}
S.~Venkatesh.
\newblock Epsilon capacity of neural networks.
\newblock {\em Proc. Conf. on Neural Networks for Computing, Snowbird, UT},
  1986.

\bibitem{Wendel62}
J.~G. Wendel.
\newblock A problem in geometric probability.
\newblock {\em Mathematica Scandinavica}, 1:109--111, 1962.

\bibitem{Winder61}
R.~O. Winder.
\newblock Single stage threshold logic.
\newblock {\em Switching circuit theory and logical design}, pages 321--332,
  Sep. 1961.
\newblock AIEE Special publications S-134.

\end{thebibliography}
\end{singlespace}

\end{document}